\documentclass[10pt,reqno,a4]{amsart}

\usepackage[latin1]{inputenc}
\usepackage{amssymb,amsmath,amscd,amsfonts,color, amsthm}
\usepackage{hyperref}

\title[CLT for certain statistics of Gram random matrices]
{A CLT for Information-theoretic statistics
of Non-centered Gram random matrices}

\author[Hachem et al.]{Walid Hachem, Malika Kharouf, Jamal Najim and Jack
W. Silverstein} \date{\today}

\newtheorem{theo}{Theorem}[section]

\newtheorem{lemma}[theo]{Lemma}

\newtheorem{prop}[theo]{Proposition}
\newtheorem{assump}{Assumption A-\hspace{-0.15cm}}
\newcommand{\R}{\mathbb{R}}
\newcommand{\C}{\mathbb{C}}
\newcommand{\E}{\mathbb{E}}
\newcommand{\vdg}{\mbox{vdiag}}
\newcommand{\bdm}{\begin{displaymath}}
\newcommand{\edm}{\end{displaymath}}
\newcommand{\bea}{\begin{eqnarray*}}
\newcommand{\eea}{\end{eqnarray*}}

\newcommand{\Rplus}{\mathbb{R}^+}

\DeclareMathOperator*{\im}{Im}
\DeclareMathOperator*{\re}{Re}
\newcommand{\bs}{\boldsymbol}

\DeclareMathOperator*{\diag}{diag}

\numberwithin{equation}{section}
\theoremstyle{remark}
\newtheorem{rem}{Remark}[section]

\DeclareMathOperator*{\tr}{Tr}

\DeclareMathOperator*{\var}{var}

\newcommand{\cvgP}[1]{\xrightarrow[#1]{\mathcal P}}
\newcommand{\cvgD}{\xrightarrow[]{\mathcal D}}

\newcommand{\lminus}{\boldsymbol{\ell}^-}
\newcommand{\lplus}{\boldsymbol{\ell}^+}
\newcommand{\dmax}{\boldsymbol{d}_{\max}}
\newcommand{\dmin}{\boldsymbol{d}_{\min}}
\newcommand{\dtmax}{\boldsymbol{\tilde{d}}_{\max}}
\newcommand{\dtmin}{\boldsymbol{\tilde{d}}_{\min}}
\newcommand{\amax}{\boldsymbol{a}_{\max}}

\begin{document}
\bibliographystyle{plain}
\begin{abstract}
  In this article, we study the fluctuations of the random variable:
$$
{\mathcal I}_n(\rho) = \frac 1N \log\det\left(\Sigma_n
  \Sigma_n^*  + \rho
  I_N \right),\quad (\rho>0)
$$
where $\Sigma_n= n^{-1/2} D_n^{1/2} X_n\tilde D_n^{1/2}
+A_n$, as the dimensions of the matrices go to infinity at the same
pace. Matrices $X_n$ and $A_n$ are respectively random and
deterministic $N\times n$ matrices; matrices $D_n$ and $\tilde D_n$
are deterministic and diagonal, with respective dimensions $N\times N$
and $n\times n$; matrix $X_n=(X_{ij})$ has centered, independent and
identically distributed entries with unit variance, either real or
complex. 

We prove that when centered and properly rescaled, the random variable
${\mathcal I}_n(\rho)$ satisfies a Central Limit Theorem and has a
Gaussian limit. The variance of ${\mathcal I}_n(\rho)$ depends on the
moment $\E X_{ij}^2$ of the variables $X_{ij}$ and also on its fourth
cumulant $\kappa= \E|X_{ij}|^4 - 2 - |\E X_{ij}^2|^2$.

The main motivation comes from the field of wireless communications,
where ${\mathcal I}_n(\rho)$ represents the mutual information of a
multiple antenna radio channel. This article closely follows the
companion article "A CLT for Information-theoretic statistics of Gram
random matrices with a given variance profile", {\em
  Ann. Appl. Probab. (2008)} by Hachem et al., however the study of
the fluctuations associated to non-centered large random matrices
raises specific issues, which are addressed here.

\end{abstract}

\maketitle
\noindent \textbf{Key words and phrases:} 
Random Matrix, Spectral measure, 
Stieltjes Transform, Central Limit Theorem.\\
\noindent \textbf{AMS 2000 subject classification:} Primary 15A52, Secondary 15A18, 60F15.

\section{Introduction} 

\subsection*{The model, the statistics, and the literature} 
Consider a $N\times n$ random matrix $\Sigma_n=(\xi_{ij}^{n})$ which 
has the expression 
\begin{equation}
\label{eq:matrix-model}
\Sigma_n= \frac{1}{\sqrt{n}} D_n^{\frac 12} X_n \tilde D_n^{\frac 12} + A_n\ ,
\end{equation}
where $A_n=(a_{ij}^{n})$ is a deterministic $N\times n$ matrix, 
$D_n$ and $\tilde D_n$ are diagonal
deterministic matrices with nonnegative entries, with respective
dimensions $N\times N$ and $n\times n$; $X_n=(X_{ij})$ is a $N\times
n$ matrix with the entries $X_{ij}$'s being centered, independent and
identically distributed (i.i.d.)  random variables with unit variance
$\E |X_{ij}|^2=1$ and finite $16^{\mathrm{th}}$ moment.

Consider the following linear statistics of the eigenvalues:
$$
{\mathcal I}_n(\rho) = \frac 1N \log \det \left( \Sigma_n \Sigma_n^* + \rho I_N \right)
=\frac 1N \sum_{i=1}^N \log (\lambda_i +\rho)\ ,
$$
where $I_N$ is the $N\times N$ identity matrix, $\rho > 0$ is a given
parameter and the $\lambda_i$'s are the eigenvalues of matrix
$\Sigma_n \Sigma_n^*$ ($\Sigma_n^*$ stands for the Hermitian adjoint
of $\Sigma_n$). This functional, known as the mutual information for
multiple antenna radio channels, is fundamental in wireless
communication as it characterizes the performance of a (coherent)
communication over a wireless Multiple-Input Multiple-Output (MIMO)
channel with gain matrix $\Sigma_n$. 
When $\Sigma_n$ follows the model described by \eqref{eq:matrix-model}, 
the deterministic matrix $A_n$ accounts for the so-called specular 
component, while $D_n$ and $\tilde D_n$ account for the correlations in 
certain bases at the receiving and emitting sides, respectively.

Since the seminal work of Telatar \cite{Tel99}, the study of the
mutual information ${\mathcal I}_n(\rho)$ of a MIMO channel (and other
performance indicators) in the regime where the dimensions of the gain
matrix grow to infinity at the same pace has turned to be extremely
fruitful. However, non-centered channel matrices have been comparatively less
studied from this point of view, as their analysis is more difficult
due to the presence of the deterministic matrix $A_n$. First order
results can be found in Girko \cite{Gir01a,Gir01b}; Dozier and
Silverstein \cite{DozSil07b,DozSil07a} established convergence results
for the spectral measure; and the systematic study of the convergence
of ${\mathcal I}_n(\rho)$ for a correlated Rician channel has been
undertaken by Hachem et al. in \cite{HLN07,Dumont10}, etc. The
fluctuations of ${\mathcal I}_n$ are important as well, for the
computation of the outage probability of a MIMO channel for instance.
With the help of the replica method, Taricco \cite{Tar06,Tar08}
provided a closed-form expression for the asymptotic variance of
${\mathcal I}_n$ when the elements of $X_n$ are Gaussian. 

The purpose of this article is to establish a Central
Limit Theorem (CLT) for ${\mathcal I}_n(\rho)$ in the following regime
$$
n\rightarrow \infty \quad \textrm{and}\quad 
0<\liminf \frac Nn \le \limsup \frac Nn <\infty\ ,
$$
(simply denoted by $n\rightarrow \infty$ in the sequel) under mild
assumptions for matrices $X_n$, $A_n$, $D_n$ and $\tilde D_n$.

The contributions of this article are twofold. From a wireless
communication perspective, the fluctuations of ${\mathcal I}_n$ are
established, regardless of the Gaussianity of the entries and the
CLT conjectured by Tarrico is fully proven. Also, this article
concludes a series of studies devoted to Rician MIMO channels,
initiated in \cite{HLN07} where a deterministic equivalent of the
mutual information was provided, and continued in \cite{Dumont10}
where the computation of the ergodic capacity was addressed and an
iterative algorithm proposed.

From a mathematical point of view, the study of the fluctuations of
${\mathcal I}_n$ is the first attempt (up to our knowledge) to
establish a CLT for a linear statistics of the eigenvalues of a Gram
non-centered matrix (so-called signal plus noise model in
\cite{DozSil07b,DozSil07a}). It complements (but does not supersede) the CLT established in
\cite{HLN08} for a centered Gram matrix with a given variance profile.
The fact that matrix $\Sigma_n$ is non-centered
($\mathbb{E}\, \Sigma_n=A_n$) raises specific issues, from a different
nature than those addressed in close-by results
\cite{AndZei06,BaiSil04,HLN08}, etc. These issues arise from the
presence in the computations of bilinear forms $u_n^* Q_n(z)\, v_n$
where at least one of the vectors $u_n$ or $v_n$ is
deterministic. Often, the deterministic vector is related to the
columns of matrix $A_n$, and has to be dealt with in such a way that
the assumption over the spectral norm of $A_n$ is exploited.  

Another important contribution of this paper is to establish the CLT
regardless of specific assumptions on the real or complex nature of
the underlying random variables. It is in particular {\em not} assumed
that the random variables are Gaussian, neither that whenever the
random variables $X_{ij}$ are complex, their second moment
$\mathbb{E}X_{ij}^2$ is zero; nor is assumed that the random variables
are circular\footnote{A random variable $X\in \mathbb{C}$ is circular
  if the distribution of $X$ is equal to the distribution of $\rho X$
  for every $\rho\in \mathbb{C}$, $|\rho|=1$. This assumption is very
  often relevant in wireless communication and has an important
  consequence; it implies that all the cross moments $\mathbb{E} |X|^k
  X^{\ell}$ ($\ell\ge 1$) are zero.}. As we shall see, all these
assumptions, if assumed, would have resulted in substantial
simplifications. As a reward however, we obtain a variance expression
which smoothly depends upon $\mathbb{E}X_{ij}^2$ whose value is 1 in
the real case, and zero in the complex case where the real and
imaginary parts are not correlated.

Interestingly, the mutual information
${\mathcal I}_n$ has a strong relationship with the Stieltjes
transform $f_n(z)= \frac 1N \mathrm{Trace}(\Sigma_n \Sigma_n^*
-zI_N)^{-1}$ of the spectral measure of $\Sigma_n \Sigma_n^*$:
$$
{\mathcal I}_n(\rho) = \log \rho + \int_{\rho}^\infty \left( 
\frac 1w - f_n(-w) \right)\, {\bf d} w\ .
$$
Accordingly, the study of the fluctuations of ${\mathcal I}_n$ is also
an important step toward the study of general linear statistics of
$\Sigma_n \Sigma_n^*$'s eigenvalues which can be expressed via the
Stieltjes transform:
$$
\frac 1N \mathrm{Trace}\ h(\Sigma_n \Sigma_n) 
=  \frac 1N \sum_{i=1}^N h(\lambda_i) = 
-\frac 1{2\mathbf{i}\pi} \oint_{\mathcal C} h(z) f_n(z)\, {\bf d}z\ ,
$$
for some well-chosen contour ${\mathcal C}$ (see for instance \cite{BaiSil04}).

Fluctuations for particular linear statistics (and general classes of
linear statistics) of large random matrices have been widely studied:
CLTs for Wigner matrices can be traced back to Girko \cite{Gir75} (see
also \cite{Gir03}). Results for this class of matrices have also been
obtained by Khorunzhy et al.  \cite{KKP96}, Boutet de Monvel and
Khorunzhy \cite{BouKho98}, Johansson \cite{Joh98}, Sinai and Sochnikov
\cite{SinSos98}, Soshnikov \cite{Sos00}, Cabanal-Duvillard
\cite{Cab01}, Guionnet \cite{Gui02}, Anderson and Zeitouni
\cite{AndZei06}, Mingo and Speicher \cite{MinSpe06}, Chatterjee
\cite{Cha09}, Lytova and Pastur \cite{LytPas09}, etc. The case of Gram
matrices has been studied in Arharov \cite{Arh71}, Jonsson
\cite{Jon82}, Bai and Silverstein \cite{BaiSil04}, Hachem et al.
\cite{HLN08}, and also in \cite{LytPas09,MinSpe06,Cha09}. 
Fluctuation results dedicated to wireless communication applications
have been developed in the centered case ($A_n=0$) by Debbah and
M\"uller \cite{DebMul03} and Tulino and Verd\`u \cite{TulVer04} (based
on Bai and Silverstein \cite{BaiSil04}), Hachem et
al. \cite{HLN08Ieee} (for Gaussian entries) and \cite{HLN08}. Other
fluctuation results either based on the replica method or on
saddle-point analysis have been developed by Moustakas, Sengupta and
coauthors \cite{MSS03,SenMit00pre}, and Tarrico \cite{Tar06,Tar08}.

\subsection*{Presentation of the results}
We first introduce the fundamental equations needed to express the deterministic approximation of the mutual information and the variance in the CLT.

\noindent {\em Fundamental equations, deterministic equivalents.} We
collect here resuls from \cite{HLN07}. The following system of
equations
\begin{equation}\label{eq:fundamental}
\left\{
  \begin{array}{ccc}
\delta_n(z) & = & \frac 1n \mathrm{Tr}\, D_n
\left( -z(I_N+\tilde \delta_n(z) D_n) + A_n(I_n + \delta_n(z) \tilde D_n)^{-1} A_n^*
\right)^{-1}
\\
\tilde \delta_n(z) & = & \frac 1n \mathrm{Tr}\, \tilde D_n 
\left( -z(I_n+\delta_n(z)\tilde D_n) + A^*_n(I_N + \tilde \delta_n(z) D_n)^{-1} A_n
\right)^{-1}
\end{array}
\right. ,\quad z\in \mathbb{C}- \mathbb{R}^+
\end{equation}
admits a unique solution $(\delta_n, \tilde \delta_n)$ in the class of
Stieltjes transforms of nonnegative measures\footnote{In fact,
  $\delta_n$ is the Stieltjes transform of a measure with total mass
  equal to $n^{-1}\mathrm{Tr} D_n$ while $\tilde \delta_n$ is the
  Stieltjes transform of a measure with total mass equal to
  $n^{-1}\mathrm{Tr} \tilde D_n$.} with support in $\mathbb{R}^+$.
Matrices $T_n(z)$ and $\tilde T_n(z)$ defined by
\begin{equation}\label{eq:def-T}
\left\{
\begin{array}{ccc}
T_n(z) & = & 
\left( -z(I_N+\tilde \delta_n(z) D_n) + A_n(I_n + \delta_n \tilde D_n)^{-1} A_n^*
\right)^{-1}
\\
\tilde T_n(z) & = & \left( -z(I_n+\delta_n(z)\tilde D_n) + A^*_n(I_N + \tilde \delta_n D_n)^{-1} A_n
\right)^{-1}
\end{array}
\right. 
\end{equation}
are approximations of the resolvent $Q_n(z) = (\Sigma_n \Sigma_n^* -zI_N)^{-1}$
and the co-resolvent $\tilde Q_n(z) = (\Sigma_n^* \Sigma_n -zI_N)^{-1}$ in the
sense that ($\xrightarrow[]{a.s.}$ stands for almost sure
convergence):
$$
\frac 1n \mathrm{Tr}\, \left( Q_n(z) -T_n(z)\right) 
\xrightarrow[n\rightarrow \infty]{a.s.} 0, \quad z \in \C - \R^+ 
$$
which readily gives a deterministic approximation of the Stieltjes
transform $N^{-1} \mathrm{Tr}\, Q_n(z)$ of the spectral measure of $\Sigma_n
\Sigma_n^*$ in terms of $T_n$ (and similarly for $\tilde Q_n$ and $\tilde
T_n$). Also proven in \cite{HLNV10pre} is the convergence of bilinear forms
\begin{equation}\label{eq:conv-bilinear-forms}
u_n^* (Q_n(z) - T_n(z))v_n \xrightarrow[n\rightarrow \infty]{a.s.} 0,
\quad z \in \C - \R^+ 
\end{equation}
where $(u_n)$ and $(v_n)$ are sequences of $N\times 1$ deterministic
vectors with bounded Euclidean norms, which complements the
picture of $T_n$ approximating $Q_n$.

Matrices $T_n=(t_{ij}\
; 1\le i,j\le N)$ and $\tilde T_n= (\tilde t_{ij}\ ; 1\le i,j \le n)$
will play a fundamental role in the sequel and enable us to express a
deterministic equivalent to $\mathbb{E} {\mathcal I}_n(\rho)$. Define
$V_n(\rho)$ by:
\begin{multline}\label{eq:equivalent}
V_n(\rho) = \frac 1N \log \det 
\left( \rho(I_N+\tilde \delta_n D_n) + 
A_n(I_n + \delta_n \tilde D_n)^{-1} A_n^* \right) \\
+\frac 1N \log\det(I_n+\delta_n \tilde D_n) 
-\frac{\rho n }{N}  \delta_n \tilde \delta_n\ , 
\end{multline}
where $\delta_n$ and $\tilde \delta_n$ are evaluated at
$z=-\rho$.  Then the difference $\mathbb{E}\, {\mathcal I}_n(\rho)
-V_n(\rho)$ goes to zero as $n\rightarrow \infty$.

In order to study the fluctuations $N({\mathcal I}_n(\rho) -
V_n(\rho))$ and to establish a CLT, we study separately the quantity
$N({\mathcal I}_n(\rho) - \mathbb{E} {\mathcal I}_n(\rho))$ from which
the fluctuations arise and the quantity $N(\mathbb{E} {\mathcal
  I}_n(\rho) - V_n(\rho) )$ which yields a bias.


\noindent  {\em The fluctuations.} In every case where the fluctuations
of the mutual information have been studied, the variance of $N\left(
  {\mathcal I}_n(\rho) -V_n(\rho)\right)$ always proved to take a
remarkably simple closed-form expression (see for
instance \cite{MSS03,Tar08,TulVer04} and in a more mathematical
flavour \cite{HLN08Ieee,HLN08}). The same phenomenon again occurs for
the matrix model $\Sigma_n$ under
consideration. Drop the subscripts $N,n$ and let
\begin{equation}\label{eq:def-gamma}
\gamma =\frac 1n \mathrm{Tr}\, D T D T\ , 
\ \tilde \gamma =\frac 1n \mathrm{Tr}\, \tilde D \tilde T \tilde D \tilde T\ ,
\ \underline{\gamma} =\frac 1n \mathrm{Tr}\, D T D \bar{T}\ , 
\ \underline{\tilde \gamma} =\frac 1n \mathrm{Tr}\, \tilde D \tilde T \tilde D \bar{\tilde T}\ ,
\end{equation} 
where $\bar{M}$ stands for the (elementwise) conjugate of matrix $M$.
Let 
$$
\vartheta=\E (X_{ij})^2\quad \textrm{and}\quad \kappa= \E |X_{ij}|^4 -2 - |\vartheta|^2\ .
$$ 
Let 
\begin{multline*} 
\Theta_n = -\log \left( 
\left( 1- \frac 1n \mathrm{Tr}\, D^{\frac 12} T A(I +\delta \tilde D)^{-1} 
\tilde D  (I +\delta \tilde D)^{-1}  A^* T D^{\frac 12}
\right)^2 - \rho^2 \gamma \tilde \gamma \right)\\
-\log \left( 
\left| 1- \vartheta \frac 1n \mathrm{Tr}\, D^{\frac 12} \bar{T} \bar{A}(I + \delta \tilde D)^{-1} 
\tilde D  (I + \delta \tilde D)^{-1}  A^* T D^{\frac 12}
\right|^2
-|\vartheta|^2\rho^2 \underline{\gamma}\, \underline{\tilde \gamma} 
\right) \\
+ \kappa \, \frac{\rho^2}{n^2} \sum_i d_i^2 t_{ii}^2 \sum_j \tilde d_j^2 \tilde t_{jj}^2\ ,
\end{multline*}
where $d_i=[D_n]_{ii}$, $\tilde d_j=[\tilde D_n]_{jj}$, and all the
needed quantities are evaluated at $z=-\rho$. The CLT can then be expressed 
as:
$$
\frac N{\sqrt{\Theta_n}}\left( {\mathcal I}_n -\mathbb{E}{\mathcal I}_n
\right) \xrightarrow[n\rightarrow \infty]{\mathcal D} {\mathcal
  N}(0,1)\ ,
$$
where $\xrightarrow[]{\mathcal D}$ stands for convergence in
distribution. Although complicated at first sight, variance
$\Theta_n$ encompasses the case of standard real random variables
($\vartheta=1$), standard complex random variables ($\vartheta=0$) and
all the intermediate cases $0<| \vartheta | < 1$. Moreover,
$\Theta_n$ often takes simpler forms if the variables are Gaussian,
real, etc. (see for instance Remark \ref{rem:simpler-forms}).


\noindent {\em The bias.} When the entries of $X_n$ are complex Gaussian 
with independent and identically distributed real and imaginary parts, 
$\kappa=\vartheta=0$, and it has already been proven in \cite{Dumont10} that 
$\mathbb{E}\, {\mathcal I}_n(\rho) -V_n(\rho) = {\mathcal O}(n^{-2})$. 
When any of $\kappa$ or $\vartheta$ is non zero, a bias term 
${\mathcal B}_n(\rho) \neq 0$ appears in the sense that  
$$
N
\left(\mathbb{E} {\mathcal I}_n(\rho) - V_n(\rho) \right) 
- {\mathcal B}_n(\rho) \xrightarrow[n\rightarrow \infty]{} 0\ . 
$$
We establish the existence of this bias and provide its expression in the
case where $A = 0$.

\subsection*{Outline of the article}

In Section \ref{sec:CLT}, we provide the main assumptions and state
the main results of the paper: Definition of the variance $\Theta_n$
and asymptotic fluctuations of $N \left( {\mathcal I}_n(\rho) - \E
  {\mathcal I}_n(\rho) \right)$ (Theorem \ref{th:CLT}), asymptotic
bias of $N \left( \E {\mathcal I}_n(\rho) - V_n(\rho) \right)$
(Proposition \ref{prop:bias}). Notations, important estimates and classical
results are provided in Section \ref{sec:notations}. Sections
\ref{sec:proof-clt}, \ref{sec:proof-clt-2} and \ref{sec:proof-clt-3}
are devoted to the proof of Theorem \ref{th:CLT}. In Section
\ref{sec:proof-clt}, the general framework of the proof is exposed; in
Section \ref{sec:proof-clt-2}, the central part of the CLT and of the
identification of the variance are established; remaining proofs are
provided in Section \ref{sec:proof-clt-3}. Finally, proof of Proposition
\ref{prop:bias} (bias) is provided in Section \ref{sec:proof-bias}.

\subsection*{Acknowlegment}
Hachem and Najim's work was partially supported by the French Agence Nationale 
de la Recherche, project SESAME n$^\circ$ ANR-07-MDCO-012-01.
Silverstein's work was supported by the 
U.S.~Army Research Office under Grant W911NF-09-1-0266.

\section{The Central Limit Theorem for ${\mathcal I}_n(\rho)$}
\label{sec:CLT}
\subsection{Notations, assumptions and first-order results}
Let $\mathbf{i}=\sqrt{-1}$.  As usual, $\Rplus = \{ x \in \R \ : \ x \geq
0 \}$. Denote by $\cvgP{ }$ the convergence in probability of random variables 
and by $\cvgD$ the convergence in distribution of probability measures. 
Denote by $\mathrm{diag}(a_i;\,1\le i\le k)$ the $k\times k$ diagonal matrix
whose diagonal entries are the $a_i$'s.  Element $(i,j)$ of matrix $M$
will be either denoted $m_{ij}$ or $[M]_{ij}$ depending on the
notational context. If $M$ is a $n\times n$ square matrix,
$\mathrm{diag}(M)=\mathrm{diag}(m_{ii}; 1\le i\le n)$. Denote by $M^T$
the matrix transpose of $M$, by $M^*$ its Hermitian adjoint, by
$\bar{M}$ the (elementwise) conjugate of matrix $M$, by $\tr (M)$ its
trace and $\det(M)$ its determinant (if $M$ is square).  When dealing
with vectors, $\|\cdot\|$ will refer to the Euclidean norm. In the case of
matrices, $\|\cdot\|$ will refer to the spectral norm. 
We shall denote by $K$ a generic constant that does not depend on $n$ and 
that might change from a line to another. If $(u_n)$ is a
sequence of real numbers, then $u_n={\mathcal O}(v_n)$ stands for
$|u_n|\le K|v_n|$ where constant $K$ does not depend on $n$.

Recall that 
\begin{equation}\label{eq:def-Sigma}
\Sigma_n = \frac 1{\sqrt{n}} D_n^{1/2} X_n \tilde D_n^{1/2} +
A_n\ ,
\end{equation}
denote $D_n=\mathrm{diag}(d_i,\ 1\le i\le N)$ and $\tilde
D_n=\mathrm{diag}(\tilde d_j,\ 1\le j\le n)$. When no confusion can
occur, we shall often drop subscripts and superscripts $n$ for
readability. Recall also that the asymptotic regime of interest 
is: 
$$
n\rightarrow \infty \quad \textrm{and}\quad 
0\ <\ \liminf \frac Nn \ \le\ \limsup \frac Nn \ <\infty\ ,
$$
and will be simply denoted by $n\rightarrow \infty$ in the sequel. We can assume without loss 
of generality that there exist nonnegative real numbers $\lminus$ and $\lplus$
such that:
\begin{equation}\label{eq:asymptotic}
0\ <\ \lminus \ \le\ \frac Nn \ \le\ \lplus \ <\ \infty\quad\textrm{as}\quad n\rightarrow \infty\ .
\end{equation}

\begin{assump}
\label{ass:X}
The random variables $(X_{ij}^n\ ;\ 1\le i\le N,\,1\le j\le n\,,\, n\ge1)$ 
are complex, independent and identically
distributed. They satisfy 
$$
\E X_{11}^n =0, \quad
\E|X_{11}^n|^2=1 \quad \mathrm{and} \quad 
\E|X_{11}^n|^{16}<\infty \ . 
$$
\end{assump}

\begin{rem} (Gaussian distributions) If $X_{11}$ is a standard
  complex or real Gaussian random variable, then $\kappa=0$. More
  precisely, in the complex case, $\re(X_{11})$ and
  $\im(X_{11})$ are independent real Gaussian random
  variables, then $\vartheta=\kappa=0$; in the real case, then
  $\vartheta=1$ while $\kappa=0$.
\end{rem}

\begin{assump}
\label{ass:A} The family of deterministic $N\times n$ complex matrices 
$(A_n,n\ge 1)$ is bounded for the spectral norm:
$$
\amax= \sup_{n\ge 1} \|A_n\| <\infty\ .
$$
\end{assump}

\begin{assump}\label{ass:D} The families of real deterministic $N\times N$ and $n\times n$ 
  matrices $(D_n)$ and $(\tilde D_n)$ are diagonal with non-negative
  diagonal elements, and are bounded for the spectral norm as
  $n\rightarrow\infty$:
$$
\dmax = \sup_{n\ge 1} \|D_n\| <\infty\quad\textrm{and}\quad
\dtmax = \sup_{n\ge 1} \|\tilde D_n\| <\infty\ .
$$  
Moreover, 
$$
\dmin  = \inf_{n} \frac 1n \tr D_n >0\quad \textrm{and}\quad 
\dtmin = \inf_{n} \frac 1n \tr \tilde D_n >0\ .
$$
\end{assump}

\begin{theo}[First order results - \cite{HLN07,Dumont10}]
\label{th:first-order}  
  Consider the $N\times n$ matrix $\Sigma_n$ given by
  \eqref{eq:def-Sigma} and assume that {\bf A-\ref{ass:X}}, {\bf
    A-\ref{ass:A}} and {\bf A-\ref{ass:D}} hold true. Then, the system
  \eqref{eq:fundamental} admits a unique solution $(\delta_n, \tilde
  \delta_n)$ in the class of Stieltjes transforms of nonnegative
  measures. Moreover, 
\[
\frac 1n \mathrm{Tr}\, \left( Q_n(z) -T_n(z)\right) 
\xrightarrow[n\rightarrow \infty]{a.s.} 0 \quad \text{and} \quad 
\frac 1n \mathrm{Tr}\, \left( \tilde Q_n(z) - \tilde T_n(z)\right) 
\xrightarrow[n\rightarrow \infty]{a.s.} 0 \quad \text{for any } 
z \in \C - \R^+ .  
\]
\end{theo}

\subsection{The Central Limit Theorem} 
In this section, we state the CLT then provide the asymptotic bias in some 
particular cases.

\begin{theo}[The CLT]\label{th:CLT}
Consider the $N\times n$ matrix $\Sigma_n$ given by \eqref{eq:def-Sigma}
and assume that {\bf A-\ref{ass:X}}, {\bf A-\ref{ass:A}} and 
{\bf A-\ref{ass:D}} hold true. Recall the definitions of $\delta$ and 
$\tilde{\delta}$ given by \eqref{eq:fundamental}, $T$ and
$\tilde T$ given by \eqref{eq:def-T}, $\gamma$, $\tilde \gamma$, 
$\underline{\gamma}$ and $\underline{\tilde \gamma}$ given by 
\eqref{eq:def-gamma}. Let $\rho>0$. All the considered quantities are 
evaluated at $z=-\rho$. Define $\Delta_n$ and $\underline\Delta_n$ as 
\[
\Delta_n = 
\left( 1- \frac 1n \tr D^{\frac 12} T A(I + \delta \tilde D)^{-2} 
\tilde D A^* T D^{\frac 12}
\right)^2 - \rho^2 \gamma \tilde \gamma 
\]
and 
\[
\underline\Delta_n = 
\left| 1- \vartheta 
\frac 1n \mathrm{Tr}\, D^{\frac 12} \bar{T} \bar{A}(I +\delta \tilde D)^{-2}
\tilde D A^* T D^{\frac 12}
\right|^2
-|\vartheta|^2\rho^2 \underline{\gamma}\, \underline{\tilde \gamma} . 
\]
Then the real numbers 
\begin{equation}\label{eq:def-variance}
\Theta_n= -\log \Delta_n -\log\underline\Delta_n 
+ \kappa \, 
\frac{\rho^2}{n^2} \sum_{i=1}^N d_i^2 t_{ii}^2 
\sum_{j=1}^n \tilde d_j^2 \tilde t_{jj}^2
\end{equation}
are well-defined and satisfy:
\begin{equation}\label{eq:estimate-variance}
0\quad  < \quad \liminf_n \Theta_n \quad \leq\quad  \limsup_n 
\Theta_n\quad 
<\quad  \infty
\end{equation}
as $n \to \infty$. Let 
$${\mathcal I}_n(\rho) = \frac 1N \log \det\left(\Sigma_n \Sigma_n^* + \rho
  I_N \right)\ ,$$ then the following convergence holds true:
$$
\frac N{\sqrt{\Theta_n}}
\left( {\mathcal I}_n(\rho) -\mathbb{E}{\mathcal I}_n(\rho)
\right) \xrightarrow[n\rightarrow \infty]{\mathcal D} {\mathcal
  N}(0,1)\ .
$$
\end{theo}

\begin{rem}\label{rem:simpler-forms}(Simpler forms for the variance) We consider here special cases where the 
  variance $\Theta_n$ takes a simpler form.
\begin{enumerate}
\item The standard complex Gaussian case. Assume that the $X_{ij}$'s
  are standard complex Gaussian random variables, i.e. that both the
  real and imaginary parts of $X_{ij}$ are independent real Gaussian
  random variables, each with variance 1/2. In this case,
  $\vartheta=\kappa=0$ and $\Theta_n$ is equal to $-\log\Delta_n$, and we in 
  particular recover the variance formula given in \cite{Tar08}.
\item The standard real case. Assume that the $X_{ij}$'s are standard
  real random variables, assume also that $A$ has real entries. Then
  $\Delta_n$ and $\underline\Delta_n$ are equal. 
\item The 'signal plus noise' model. In this case, $D_n=I_N$ and $\tilde
  D_n=I_n$, which already yields simplifications in the variance expression. 
  In the case where $\vartheta=0$, the variance is:
\[
\quad\quad\quad  
\Theta_n= - \log \left( \left( 1 - 
n^{-1} (1+\delta)^{-2} \tr T A A^* T \right)^2 - \rho^2 \gamma \tilde \gamma
\right) + \frac{\kappa \rho^2}{n^2} \sum_i d^2_i t^2_{ii} \sum_j
\tilde d^2_j \tilde t^2_{jj}\ .
\] 

As one may easily check, the first term of the variance only depends
upon the spectrum of $A A^*$. The second term however also depends
on the eigenvectors of $AA^*$ (see for instance \cite{HKKN10}).
\end{enumerate}

\end{rem}

A full study of the asymptotic bias turns out to be extremely involved
and would have substantially increased the volume of this paper. In
the following proposition, we restrict our study to two important particular 
cases: 
\begin{prop}[The bias - particular cases]
\label{prop:bias} 
Assume that the setting of Theorem \ref{th:CLT} holds true. 
\begin{enumerate}
\item[(i)] If the random variables $(X_{ij}^n; i,j,n)$ are complex 
  with $\re(X_{ij}^n)$ and $\im(X_{ij}^n)$
  independent, both with distribution ${\mathcal N}(0,1/2)$, then:
$$
N\left(\mathbb{E} {\mathcal I}_n(\rho) - V_n(\rho) \right) 
= {\mathcal O}\left( \frac 1N\right)\ .
$$
\item[(ii)] If $A_n=  0$, let the quantities $\gamma$ and $\tilde\gamma$
be evaluated at $z=-\rho$ and consider
\begin{equation}
\label{def:bias}
{\mathcal B}_n = - \frac{\kappa}{2} \rho^2 \gamma \tilde\gamma 
+ \frac 12 \log( 1 - |\vartheta |^2 \rho^2 \gamma \tilde\gamma )  . 
\end{equation}
Then 
$$
N \left( \E \mathcal{I}_n(\rho) - V_n(\rho)\right) - {\mathcal B}_n 
\xrightarrow [n \rightarrow \infty]{} 0.
$$
\end{enumerate}
\end{prop}

\begin{rem}
Observe that $T(z) = [ -z(I + \tilde\delta(z) D ) ]^{-1}$ and 
$\tilde T(z) = [ -z(I + \delta(z) \tilde D ) ]^{-1}$ when $A = 0$. 
It is interesting to notice that ${\mathcal B}_n$ coincides in that case with 
$- 0.5 \ \times$ the sum of the two last terms at the right hand side 
(r.h.s.) of \eqref{eq:def-variance}. 
\end{rem} 

Proof of Proposition \ref{prop:bias} is deferred to Section 
\ref{sec:proof-bias}.

\section{Notations and classical results}\label{sec:notations}

\subsection{Further notations}
We denote by $Y$ the $N\times n$ matrix $n^{-1/2} D^{1/2} X \tilde D^{1/2}$; 
by $(\eta_j)$, $(a_j)$ and $(y_j)$ the columns of matrices $\Sigma$, $A$ and
$Y$. Denote by $\Sigma_j$, $A_j$, and $Y_j$ the matrices $\Sigma$, $A$, and
$Y$ where column $j$ has been removed. The associated resolvent 
is $Q_j(z)=(\Sigma_j \Sigma_j^* - z I_N)^{-1}$. We shall often write
$Q$, $Q_j$, $T$ for $Q(z)$, $Q_j(z)$, $T(z)$, etc.  
We denote by $\tilde D_j$ matrix $\tilde D$ where row and column $j$ have
been removed. We also denote by $A_{1:j}$ and $\Sigma_{1:j}$ the $N \times j$ 
matrices
$A_{1:j} = [a_1, \cdots, a_j ]$ and $\Sigma_{1:j} = [\eta_1, \cdots, \eta_j ]$.
Denote by $\mathbb{E}_j$
the conditional expectation with respect to the $\sigma$-field
${\mathcal F}_j$ generated by the vectors $(y_\ell,\, 1\le \ell\le j)$.
By convention, $\mathbb{E}_0=\mathbb{E}$.

We introduce here intermediate quantities of constant use in the rest of the 
paper. For $1\le j\le n$, let:
\begin{eqnarray}
  \tilde b_j(z) &
=&  \frac {-1}{z\left( 1+a_j^* Q_j(z) a_j 
+\frac{\tilde d_j}{n} \tr D Q_j(z)\right)}\ , \nonumber \\
e_j(z) &=& \eta_j^* Q_j(z) \eta_j - 
\Bigl( \frac{\tilde d_j}{n} \tr D Q_j(z) + a_j^* Q_j(z) a_j\Bigr) 
\nonumber\\
&=& y_j^* Q_j(z) y_j -\frac{\tilde d_j}{n} \tr D Q_j(z)   
+ a_j^* Q_j(z) y_j  + y_j^* Q_j(z) a_j . 
\label{eq:def-e} 
\end{eqnarray}

\subsection{Important identities}

Recall the following classical identity for the inverse of a perturbed
matrix (see \cite[Section 0.7.4]{HorJoh90}):
\begin{equation}
\label{eq:mx-inversion}
\left( A + XRY\right)^{-1} = A^{-1} -A^{-1} X\left( R^{-1} +YA^{-1} X\right)^{-1} Y A^{-1}\ .
\end{equation}

\subsubsection*{Identities involving the resolvents} 

The following identity expresses the diagonal elements 
$\tilde q_{jj}(z) = [ \tilde Q(z) ]_{jj}$ of the co-resolvent; 
the two following ones are obtained from \eqref{eq:mx-inversion}.  
\begin{align} 
  \tilde q_{jj}(z) &= \frac{-1}{z(1+\eta_j^* Q_j(z) \eta_j)}\ , 
\label{eq:diag-coresolvent}\\
  Q(z) &= Q_j(z) - \frac{Q_j(z) \eta_j \eta_j^* Q_j(z)}
{1+\eta_j^* Q_j(z) \eta_j} \nonumber \\
&= Q_j(z) +z \tilde q_{jj}(z) Q_j(z) \eta_j \eta_j^* Q_j(z) \, 
\label{eq:woodbury} \\
  Q_j(z) &= Q(z) + \frac{Q(z) \eta_j \eta_j^* Q(z)}{1-\eta_j^* Q(z) \eta_j} 
\label{eq:Qj=f(Q,eta)} \\
1+ \eta_j^* Q_j(z) \eta_j &= \frac{1}{1 - \eta_j^* Q(z) \eta_j} 
\label{eq:1+etaQeta} 
\end{align} 
Notice that
\begin{equation}\label{eq:diff-q-b}
\tilde q_{jj}(z) = \tilde b_j(z) + z \tilde q_{jj}(z) \tilde b_j(z) e_j(z)\ .
\end{equation}
and that $0 < \tilde b_j(-\rho), \tilde q_{jj}(-\rho) < \rho^{-1}$. These 
facts will be repeatedly used in the remainder. 
A useful consequence of \eqref{eq:woodbury} is:
\begin{equation}
\eta_j^* Q(z) = \frac{\eta_j^* Q_j(z)}{1+\eta_j^* Q_j(z) \eta_j} 
= -z \tilde q_{jj}(z) \eta_j^* Q_j(z) \ .
\label{eq:woodbury2}
\end{equation}

\subsubsection*{Identities involving the deterministic equivalents $T$ and 
$\tilde T$} 
Define the $N\times N$ matrix ${\mathcal T}_j(z)$ as
\begin{equation}
\label{eq:T-cal-j}
{\mathcal T}_j(z) = \left(-z(I_N + \tilde \delta(z) D ) 
+ A_j (I_{n-1} + \delta(z) \tilde D_j)^{-1} A_j^* \right)^{-1}\ ,
\end{equation}
where $\delta$ and $\tilde \delta$ are defined in \eqref{eq:fundamental}. 
Notice that matrix ${\mathcal T}_j$ is not obtained in general by
solving the analogue of system \eqref{eq:def-T} where $A$ is replaced
with $A_j$ and when $\tilde D$ is truncated accordingly. This matrix naturally 
pops up when expressing the diagonal elements $\tilde t_{jj}$ of $\tilde T$. 
Indeed, we obtain (see Appendix \ref{expr-T_j}): 

\begin{equation}
\label{eq:t-tilde-jj}
{\tilde t}_{jj}(z) = \frac {-1}
{z\left( 1+a_j^* {\mathcal T}_j(z) a_j + \tilde d_j \delta(z) \right)} \ .
\end{equation}

Let $b$ be a given $N \times 1$ vector. The following identity is also shown
in Appendix \ref{expr-T_j}:
\begin{equation}
\label{eq:T-cal(T)_j}
-z  \tilde{t}_{\ell\ell}(z) a^*_{\ell}{\mathcal T}_{\ell}(z) b 
= \frac{a^*_{\ell} T(z) b}{1 + \tilde d_\ell \delta(z)}   . 
\end{equation}

Thanks to \eqref{eq:mx-inversion}, we also have 
\begin{equation}
\label{eq:Ttilde=f(T)}  
\tilde T(z) = - z^{-1} ( I + \delta(z) \tilde D)^{-1} 
+ z^{-1} ( I + \delta(z) \tilde D)^{-1} A^* T(z) A 
( I + \delta(z) \tilde D)^{-1} . 
\end{equation} 

\subsection{Important estimates}
We gather in this section matrix estimates which will be of constant use in 
the sequel. In all the remainder, $z$ will belong to 
the open negative real axis, and will be fixed to $z = - \rho$ until Section 
\ref{sec:proof-bias}. 

Let $A$ and $B$ be two square matrices. Then 
\begin{equation}
\label{ineq:trace}
\left| \tr (AB) \right| \leq \sqrt{\tr(AA^* )} \sqrt{\tr(BB^* )} 
\end{equation} 
When $B$ is Hermitian non negative, then a consequence of Von Neumann's
trace theorem is 
\begin{equation}
\label{eq:von-neumann-lite}
| \tr (AB) | \quad \le\quad  \|A\|\, \tr B\ .
\end{equation}

The following lemma gives an estimate for a rank-one perturbation of the 
resolvent (\cite[Lemma 6.3]{HLN08} and \cite[Lemma 2.6]{SilBai95}): 
\begin{lemma}
\label{lem:rank1-perturbation} 
The resolvents $Q$ and the perturbed resolvent $Q_j$ satisfy for $z=-\rho$: 
$$
\left|\tr \, A \left(Q - Q_j\right)\right| \le \frac{\|A\|}{\rho}
$$
for any $N\times N$ matrix $A$. 
\end{lemma}

The following results describe the asymptotic behaviour of quadratic forms 
based on the resolvent.

\begin{lemma}[Bai and Silverstein, Lemma 2.7 in \cite{BaiSil98}] 
\label{lm-approx-quadra}
Let $\bs{x}=(x_1,\cdots, x_n)$ be a $n \times 1$ vector where the $x_i$
are centered i.i.d.~complex random variables with unit variance. 
Let $M$ be a $n\times n$ deterministic complex matrix. Then for any $p\geq 2$,
there exists a constant $K_p$ for which 
\[
\mathbb{E}|\bs{x}^* M \bs{x} -\tr M|^p
\le K_p \left( \left( \mathbb{E}|x_1|^4 \tr M M^* \right)^{p/2} 
+ \mathbb{E} |x_1|^{2p} \tr (M M^*)^{p/2}
\right)
\]
\end{lemma}
\begin{rem}
There are some important consequences of the previous lemma. Let $(M_n)$
be a sequence of $n\times n$ deterministic matrices with bounded spectral
norm and $({\bs x}_n)$ be a sequence of random vectors as in the statement
of Lemma \ref{lm-approx-quadra}. Then for any $p \in [2; 8]$, 
\begin{equation} 
\label{eq:BS-lemma-2} 
\max\left( 
\mathbb{E} \left| \frac{\bs{x}_n^* M_n \bs{x}_n}{n} - 
\frac{\tr M_n}{n} \right|^p, 
\E |e_j|^p \right) 
\leq \frac{K}{n^{p/2}} 
\end{equation}
where $e_j$ is given by \eqref{eq:def-e} (the estimate 
$\mathbb{E}|e_j|^p ={\mathcal O}(n^{-p/2})$ is proven in 
Appendix \ref{bound-e}). 
\end{rem}
\begin{rem}
By replacing $\E|x_1|^{4}$ with $\max_i \E|x_i|^{4}$ and 
$\E|x_1|^{2p}$ with $\max_i \E|x_i|^{2p}$, Lemma \ref{lm-approx-quadra} can 
be extended to the case where elements of vector $\bs{x}$ are independent but 
not necessarily identically distributed \cite[Lemma B.26]{bai-sil-book}. 
Accordingly, the results of this paper remain true when the $X_{ij}$ are 
independent but not necessarily identically distributed, provided 
$\E|X_{ij}|^2 = 1$,  $\E X_{ij}^2 = \vartheta$, and 
$\E|X_{ij}|^4 - 2 - |\vartheta|^2 = \kappa$ for all $i,j$, and 
$\sup_n \max_{i,j} \E|X_{ij}|^{16} < \infty$. 
\end{rem} 

The following theorem is proven is Appendix \ref{anx:th-quadra}: 
\begin{theo}
\label{theo:quadra} 
Assume that the setting of Theorem \ref{th:CLT} holds true. Let $(u_n)$ and 
$(v_n)$ be two sequences of deterministic complex $N\times 1$ vectors bounded 
in the Euclidean norm:
$$
\sup_{n\ge 1} \max\left( \|u_n\|,\|v_n\|\right) <\infty , 
$$
and let $(U_n)$ be a sequence of deterministic $N\times N$ matrices with 
bounded spectral norms: 
$$
\sup_{n\ge 1} \| U_n \| <\infty. 
$$
Then,
\begin{enumerate}
\item
\label{quadra-sum-bounded} 
There exists a constant $K$ for which 
$$ 
\sum_{j=1}^n \mathbb{E} |u_n^* Q_j a_j|^2\ \le\ K . 
$$
\item 
\label{speed(T-Q)} 
The following holds true: 
$$
\left| \frac 1n \tr U( T - \E Q)\right| \le \frac{K}{n}\ .
$$
\item 
\label{conv-quadra} For every $p \in [ 1, 2]$, there exists a constant 
$K_p$ such that:
$$
\mathbb{E}\left| u_n^* (Q - T) v_n \right|^{2p} \le \frac{K_p}{n^p}\ .
$$
\item 
\label{eq-u(Qj-Tj)v} 
For every $p \in [1,2]$, there exists a constant $K_p$ such that:
\[
\E\left| u_n^* (Q_j - {\mathcal T}_j) v_n \right|^{2p} 
\le \frac{K_p}{n^p}\ .
\] 
\item 
\label{var(trUQ)} 
There exists a constant $K$ such that 
\[
\E\left| \tr U(Q - \E Q) \right|^2 < K .
\]
\end{enumerate}
\end{theo}

The following results stem from Lemma \ref{lm-approx-quadra} and Theorem 
\ref{theo:quadra} and will be of constant use in the sequel. 
Recalling \eqref{eq:diff-q-b} and \eqref{eq:BS-lemma-2} along with the bounds
on $\tilde q_{jj}$ and $\tilde b_j$, we have for any $p \in [2,8]$ 
\begin{equation}
\label{eq-(q-b)}
\E \left| \tilde q_{jj} - \tilde b_j \right|^p \leq \frac{K}{n^{p/2}} . 
\end{equation} 
Of course, the 
counterpart of Theorem \ref{theo:quadra} for the co-resolvent $\tilde Q$ and 
matrix $\tilde T$ holds true. In particular, taking the vectors $u_n$ and 
$v_n$ as the $j$th canonical vector of $\C^n$ yields the following estimate
for any $p \in [2,4]$: 
\begin{equation}
\label{eq:estimate-t-tilde}
\E \left| \tilde q_{jj} - \tilde t_{jj} \right|^p \leq \frac{K}{n^{p/2}} .
\end{equation}

The following two lemmas, proven in Appendices \ref{app:estimates-proof} and 
\ref{app:bounds-proof}, provide some important bounds:

\begin{lemma}
\label{lm:estimates} 
Assume that the setting of Theorem \ref{th:CLT} holds true. 
Then, the following quantities satisfy:
\begin{gather*}
 {\textstyle \frac{\dmin}{\rho + \dmax \dtmax + \amax^2}} \le 
  \delta_n \le  {\textstyle \frac{\lplus \dmax}{\rho }}\ , \quad 
{\textstyle \frac{\dtmin}{\rho + \lplus \dmax \dtmax  + \amax^2}}
\le \tilde{\delta}_n \le {\textstyle \frac{\dtmax}{\rho}}\ ,\\
{\textstyle \frac{\dmin}{(\rho + \dmax \dtmax + \amax^2)^2}} \le 
  \frac 1n \tr D T^2 \le  
{\textstyle \frac{\lplus \dmax}{\rho^2}}\ , \quad 
{\textstyle \frac{\dtmin}{(\rho + \lplus \dmax \dtmax  + \amax^2)^2}}
\le \frac 1n \tr \tilde D \tilde T^2 \le 
{\textstyle \frac{\dtmax}{\rho^2}}\ ,\\
{\textstyle \frac{\dmin^2}{\lplus (\rho +\dmax \dtmax +\amax^2)^2}} 
\le \gamma_n  \le 
{\textstyle \frac{\lplus \dmax^2}{\rho^2}}\ , \quad 
{\textstyle \frac{\dtmin^2}
           {\left(\rho +\lplus \dmax \dtmax  + \amax^2 \right)^2 }} 
\le \tilde{\gamma}_n \le
{\textstyle \frac{\dtmax^2}{\rho^2}} \\ 
{\textstyle \frac{\dmin^2}{\lplus (\rho +\dmax \dtmax +\amax^2)^2}} 
\le \frac 1n \sum_{i=1}^N d_i^2 t_{ii}^2  \le 
{\textstyle \frac{\lplus \dmax^2}{\rho^2}}, \quad 
{\textstyle \frac{\dtmin^2}
        {\left(\rho +\lplus \dmax \dtmax  + \amax^2 \right)^2 }} 
\le \frac 1n \sum_{j=1}^n \tilde d_j^2 \tilde t_{jj}^2 \le 
{\textstyle \frac{\dtmax^2}{\rho^2}} \ .
\end{gather*}
\end{lemma}
\begin{lemma}
\label{lm:var-bounds} 
Assume that the setting of Theorem \ref{th:CLT} holds true. Then 
\[
\sup_n \frac 1n \tr D^{1/2} T A(I + \delta \tilde D)^{-2} 
\tilde D A^* T D^{\frac 12} \, < \, 1 .
\]
Moreover, the sequence $(\Delta_n)$ as defined in Theorem \ref{th:CLT} 
satisfies 
\[
\Delta_n \geq 
\frac{\rho}{n\delta} \tr D T^2
\frac{\rho}{n\tilde \delta} \tr \tilde D \tilde T^2 
\] 
and 
$$
\liminf_n \Delta_n > 0\ .
$$
\end{lemma}

\subsection{Other important results}

The main result we shall rely on to establish the Central Limit Theorem
is the following CLT for martingales:
\begin{theo}[CLT for martingales, Th. 35.12 in \cite{Bil95}] 
\label{th:clt-martingales}
Let $\gamma^{(n)}_1, \gamma^{(n)}_{2}, \ldots, \gamma^{(n)}_n$ be a 
martingale difference sequence with respect to the increasing filtration 
${\mathcal F}^{(n)}_1, \ldots, {\mathcal F}^{(n)}_n$. 
Assume that there exists a sequence of real positive numbers 
${\Upsilon}_n^2$ such that 
\begin{equation}
\label{eq:clt-martingale} 
\frac{1}{\Upsilon_n^2}
\sum_{j=1}^n \E_{j-1} {\gamma_j^{(n)}}^2 \cvgP{n\to\infty} 1 \ .
\end{equation} 
Assume further that the Lyapounov condition (\cite[Section 27]{Bil95})
holds true:
$$
\exists \delta > 0, \quad
\frac{1}{\Upsilon_n^{2(1+\delta)}}
\sum_{j=1}^n \E\left| {\gamma_j^{(n)}} \right|^{2+\delta} 
\xrightarrow[n\rightarrow \infty]{} 0\ .
$$
Then $\Upsilon_n^{-1} \sum_{j=1}^n \gamma^{(n)}_j $ converges in 
distribution to 
${\mathcal N}(0,1)$. 
\end{theo}

\begin{rem} Note that if moreover $\liminf_n \Upsilon_n^2>0$, it is sufficient to prove: 
\begin{equation}\label{eq:clt-martingale-bis}
\sum_{j=1}^n \E_{j-1} {\gamma_j^{(n)}}^2 - \Upsilon_n^2 \cvgP{n\to\infty} 0 \ ,
\end{equation}
instead of \eqref{eq:clt-martingale}.
\end{rem}

We now state a covariance identity (the proof of which is
straightforward and therefore omitted) for quadratic forms based on
non-centered vectors. This identity explains to some extent the
various terms obtained in the variance.

Let $\bs{x}=(x_1,\cdots, x_N)^T$ be a $N \times 1$ vector where the
$x_i$ are centered i.i.d.~complex random variables with unit
variance. Let $\bs {y}=N^{-1/2} D^{1/2} \bs{x}$ where $D$ is a $N \times N$ 
diagonal nonnegative deterministic matrix. Let $M=(m_{ij})$ and
$P=(p_{ij})$ be $N\times N$ deterministic complex matrices and let 
$\bs{u}$ be a $N\times 1$ deterministic vector.

If $M$ is an $N\times N$ matrix, $\vdg(M)$ stands for the $N\times 1$ 
vector $[M_{11},\cdots, M_{NN}]^T$. 

Denote by $\Upsilon(M)$ the random variable:
$$
\Upsilon(M)= (\bs{y}+\bs{u})^* M (\bs{y}+\bs{u})\ .
$$
Then $ \mathbb{E}\Upsilon(M) = \frac 1N \tr DM + \bs{u}^* M \bs{u}\ $
and the covariance between $\Upsilon(M)$ and $\Upsilon(P)$ is:
\begin{eqnarray}\label{eq:identite-fondamentale}
  \lefteqn{\mathbb{E}\left[ \left(\Upsilon(M) - \mathbb{E} \Upsilon(M) \right) 
      \left(\Upsilon(P) - \mathbb{E} \Upsilon(P) \right)\right]} \nonumber \\
  &=& \frac 1{N^2}  \tr  (MDPD) 
+ \frac 1N \left(\bs{u}^* M D P \bs{u} + \bs{u}^*P D M \bs{u}\right)\nonumber \\
  && \quad + 
\frac{|\mathbb{E}[x_1^2]|^2}{N^2} \tr(M D P^T D) 
+ \frac {\mathbb{E}[x_1^2]}N \bs{u}^* P D M^T \bar{\bs{u}} 
  + \frac{\E[\bar{x}_1^2]}N \bs{u}^T M^T D P \bs{u}\nonumber \\
  && \quad \quad + \frac{\E[|x_1|^2x_1]}{N^{3/2}}
\left(\bs{u}^* P D^{3/2} \vdg(M) + \bs{u}^* M D^{3/2} \vdg(P)\right)\nonumber \\
  && \quad\quad\quad + \frac{\E[|x_1|^2\bar{x}_1]}{N^{3/2}}\left(\vdg(P)^T D^{3/2} M \bs{u} + \vdg(M)^T D^{3/2} P \bs{u}\right)\nonumber \\   
  && \quad\quad\quad\quad + \frac{\kappa}{N^2} 
\sum^{N}_{i=1} d^2_{ii} m_{ii} p_{ii}\ ,
\end{eqnarray}
where $\kappa = \mathbb{E}|x_1|^4 - 2 - |\E x^2_1|^2$.

\begin{rem} Identity \eqref{eq:identite-fondamentale} is the
  cornerstone for the proof of the CLT; it is the counterpart of
  Identity (1.15) in \cite{BaiSil04}. The complexity of Identity
  \eqref{eq:identite-fondamentale} with respect to \cite[Identity
  (1.15)]{BaiSil04} lies in 8 extra terms and stems from two elements:
\begin{enumerate}
\item The fact that matrix $\Sigma$ is non-centered.
\item The fact that the random variables $X_{ij}$'s are either real
  and complex with no particular assumption on their second moment
 (in particular, $\mathbb{E} X_{ij}^2$ can be non zero in the complex case).
\end{enumerate}
It is this identity which induces to a large extent all the
computations in the present article.
\end{rem}

\section{Proof of Theorem \ref{th:CLT} (part I)}\label{sec:proof-clt}
\begin{center} {\em Decomposition of
    ${\mathcal I}_n - \mathbb{E}{\mathcal I}_n$, Cumulant and cross-moments 
terms in the variance}
\end{center}

\subsection{Decomposition of ${\mathcal I}_n - \mathbb{E}{\mathcal I}_n$ 
as a sum of martingale differences}
Denote by 
$$\Gamma_j = 
\frac{\eta_j^* Q_j \eta_j - \left( \frac{\tilde d_j}{n} \tr D Q_j + a_j^* Q_j a_j\right)}{1 + \frac{\tilde d_j}{n} \tr D Q_j + a_j^* Q_j a_j}\ .
$$
With this notation at hand, the decomposition of 
${\mathcal I}_n - \mathbb{E}{\mathcal I}_n$ as 
\begin{equation}\label{eq:split-martingale}
{\mathcal I}_n - \mathbb{E}{\mathcal I}_n =  
\sum_{j=1}^n (\E_j -\E_{j-1}) \log (1+\Gamma_j)
\end{equation}
follows verbatim from \cite[Section 6.2]{HLN08}. Moreover, it is a matter of 
bookkeeping to establish the following (cf. \cite[Section 6.4]{HLN08}):
\begin{equation}\label{eq:split-martingale-bis}
\sum_{j=1}^n \E_{j-1} \left( (\E_j -\E_{j-1}) \log (1+\Gamma_j) \right)^2
- \sum_{j=1}^n \E_{j-1} (\E_j \Gamma_j)^2 \cvgP{n\to\infty} 0 \ . 
\end{equation}
Hence, the details are omitted. 
In view of 
Theorem \ref{th:clt-martingales}, 
Eq. \eqref{eq:clt-martingale-bis}, \eqref{eq:split-martingale} and
\eqref{eq:split-martingale-bis}, the CLT will be established if one
proves the following 3 results:
\begin{enumerate}
\item (Lyapounov condition)
$$
\exists \delta > 0, \quad \sum_{j=1}^n \E\left| \E_j \Gamma_j  \right|^{2+\delta} 
\xrightarrow[n\rightarrow \infty]{} 0\ ,
$$
\item (Martingale increments and variance)
$$
\sum_{j=1}^n \E_{j-1} (\E_j \Gamma_j)^2 - \Theta_n \cvgP{n\to\infty} 0 \ .
$$
\item (estimates over the variance)
$$
0<\liminf_n \Theta_n \le \limsup_n \Theta_n <\infty
$$
\end{enumerate}

It is straightforward (and hence omitted) to verify Lyapounov
condition. The convergence toward the variance is the cornerstone of
the proof of the CLT: The rest of this section together with much of
Section \ref{sec:proof-clt-2} are devoted to establish it. The
estimates over the variance $\Theta_n$, also central to apply Theorem
\ref{th:clt-martingales}, are established in Section
\ref{sec:bound-theta}.

Notice that 
$\E_{j-1} (\E_j \Gamma_j)^2= \E_{j-1} (\E_j \rho\tilde b_j e_j)^2$. 
We prove hereafter that 
\begin{equation}\label{eq:master-convergence-2}
\sum_{j=1}^n \E_{j-1} (\E_j \rho \tilde b_j e_j)^2 
- \sum_{j=1}^n \rho^2 \tilde t^2_{jj} 
\E_{j-1} (\E_j e_j)^2 \cvgP{n\to\infty} 0 \ . 
\end{equation}
The inequality $\E|\tilde b_j - \tilde t_{jj}|^2 \leq 2 
\E|\tilde b_j - \tilde q_{jj}|^2 + 2 \E|\tilde q_{jj} - \tilde t_{jj}|^2$ 
in conjunction with Estimates \eqref{eq-(q-b)} and
\eqref{eq:estimate-t-tilde} yield 
$\E|\tilde b_j - \tilde t_{jj}|^2 = {\mathcal O}(n^{-1})$. Moreover, 
\begin{multline*}
\E\left| \E_{j-1} (\E_j \rho \tilde b_j e_j)^2 -  
  \E_{j-1} (\rho \tilde t_{jj} \E_j e_j)^2\right| \leq 
\E\left| \left( \E_j \rho \tilde b_j e_j \right)^2 - 
 \left( \E_j \rho \tilde t_{jj} e_j \right)^2 \right| \\
= 
\E\left| 
\left( \E_j ( \rho \tilde b_j - \rho \tilde t_{jj} ) e_j \right)  
\left( \E_j ( \rho \tilde b_j + \rho \tilde t_{jj} ) e_j \right)  
\right|  
\leq \E\left| 
( \rho \tilde b_j - \rho \tilde t_{jj} ) e_j   
\left( \E_j ( \rho \tilde b_j + \rho \tilde t_{jj} ) e_j \right)  
\right|  \\ 
\leq K n^{-3/2} 
\end{multline*} 
using Cauchy-Schwarz inequality and \eqref{eq:BS-lemma-2}. This implies 
\eqref{eq:master-convergence-2}. Let 
$\varsigma = \E (|X_{11}^2| X_{11})$. 
Using Identity \eqref{eq:identite-fondamentale}, we develop the quantity 
$\E_{j-1} (\E_j e_j)^2$:
 
\begin{eqnarray}
  \lefteqn{  \sum_{j=1}^n \rho^2 \tilde t^2_{jj}
    \E_{j-1} (\E_j e_j)^2 }\nonumber \\
&=& \frac{\kappa}{n^2} \sum_{j=1}^n \rho^2 \tilde d_j^2 \tilde t_{jj}^2 
\sum_{i=1}^N d_{i}^2  [ \E_j Q_j ]_{ii}^2 \nonumber \\ 
&\phantom{=}& 
+ \frac{4}{n} \sum_{j=1}^n \rho^2 \tilde d_j^{3/2} \tilde t_{jj}^2 
\re\left( \varsigma \frac{a_j^* (\E_j Q_j) D^{3/2} \vdg( \E_j Q_j)}
{\sqrt{n}} \right)\nonumber  \\
&\phantom{=}& 
+ \frac{1}{n} \sum_{j=1}^n \rho^2 \tilde t_{jj}^2 
\left( 
\frac{\tilde d_j^2}{n}  \tr (\E_j Q_j) D (\E_j Q_j) D + 
2 \tilde d_j a_j^* (\E_j Q_j) D (\E_j Q_j) a_j 
\right) \nonumber \\
&\phantom{=}& 
+ \frac{1}{n} \sum_{j=1}^n \rho^2 \tilde t_{jj}^2 
\left( 
|\vartheta|^2 \frac{\tilde d_j^2}{n} \tr (\E_j Q_j) D (\E_j \bar Q_j) D 
+ 2 \re \left( 
  \vartheta \tilde d_j a_j^* (\E_j Q_j) D (\E_j \bar Q_j) \bar a_j \right)
\right) \nonumber \\
&\stackrel{\triangle}{=}&  
 \sum_{j=1}^n \chi_{1j} + \sum_{j=1}^n \chi_{2j} + \sum_{j=1}^n \chi_{3j} 
+ \sum_{j=1}^n \chi_{4j} \ . \nonumber 
\end{eqnarray} 

\subsection{Key lemmas for the identification of the variance}
The remainder of the proof of Theorem \ref{th:CLT} is devoted to find
deterministic equivalents for the terms $\sum_{j=1}^n \chi_{\ell j}$ 
for $\ell=1,2,3,4$. 

\begin{lemma}
\label{lemma:chi1} 
Assume that the setting of Theorem \ref{th:CLT} holds true, then:
$$
\sum_{j=1}^n \chi_{1j} 
- \frac{\kappa \rho^2}{n^2} \sum_{i=1}^N \sum_{j=1}^n 
d_i^2 t_{ii}^2  \tilde d_j^2 \tilde t_{jj}^2 \cvgP{n\to\infty} 0\ .
$$
\end{lemma}

\begin{proof}
Write
\begin{multline*}
\frac 1n \sum_{i=1}^N d_{i}^2  [ \E_j Q_j ]_{ii}^2 - 
\frac 1n \sum_{i=1}^N d_{i}^2  [ \E_j Q_j ]_{ii} t_{ii} = \\ 
\frac 1n \sum_{i=1}^N d_i^2 [\E_{j} Q_j]_{ii} \E_{j}([Q_j]_{ii}- [Q]_{ii}) 
+
\frac 1n \sum_{i=1}^N d_i^2 [\E_{j} Q_j]_{ii} ( [ \E_{j}Q]_{ii}- t_{ii})
= \varepsilon_{1,j} + \varepsilon_{2,j} . 
\end{multline*}
The term 
$| \varepsilon_{1,j} | = n^{-1}| \E_j[\tr D^2 \diag(\E_{j} Q_j) (Q_j - Q)]|$ 
is of order ${\mathcal O}(n^{-1})$ thanks to Lemma 
\ref{lem:rank1-perturbation}. Moreover, 
$\E | \varepsilon_{2,j} | = {\mathcal O}(n^{-1/2})$ by the analogue of 
\eqref{eq:estimate-t-tilde} for the diagonal elements of the resolvent. Hence, 
$$
\sum_{j=1}^n \chi_{1j} - \frac{\kappa \rho^2}{n} 
\sum_{j=1}^n \tilde d_j^2 \tilde t_{jj}^2 \left( \frac 1n 
      \sum_{i=1}^N d_i^2 [ \E_{j} Q_j]_{ii} t_{ii} 
    \right)\cvgP{n\to\infty} 0\ .
$$
Iterating the same arguments, we can replace the remaining term
$\E_{j} [Q_j]_{ii}$ by $t_{ii}$ to obtain the desired result. 
\end{proof}

\begin{lemma}
\label{lemma:chi2} 
Assume that the setting of Theorem \ref{th:CLT} holds true. Then:
$$
\sum_{j=1}^n \chi_{2j} \cvgP{n\to\infty} 0\ .
$$
\end{lemma}

\begin{proof} 
We have 
\begin{eqnarray}
\E \left| \sum_{j=1}^n \chi_{2j} \right| &\le& 
\frac Kn \sum_{j=1}^n 
\E\left| a_j^* (\E_j Q_j) D^{3/2} \frac{ \vdg (Q_j)}{\sqrt{n}} \right| 
\nonumber\\
&\le& \frac Kn \sum_{j=1}^n 
\E\left| a_j^* Q_j D^{3/2} \frac{\vdg (T)}{\sqrt{n}} \right| 
+ \frac Kn \sum_{j=1}^n 
\E\left| a_j^* (\E_j Q_j) D^{3/2} \frac{\vdg (Q-T)}{\sqrt{n}} \right|
\nonumber \\
&&\quad
+
\frac Kn \sum_{j=1}^n 
\E\left| a_j^* (\E_j Q_j) D^{3/2} \frac{\vdg (Q_j-Q)}{\sqrt{n}} \right|\ . 
\label{eq:chi2}
\end{eqnarray}
The first term satisfies 
$$
\sum_{j=1}^n  \E\left| a_j^* Q_j D^{3/2} \frac{\vdg (T)}{\sqrt{n}} \right| 
\le \sqrt{n} \left(  \sum_{j=1}^n 
\E\left| a_j^* Q_j D^{3/2} \frac{\vdg (T)}{\sqrt{n}} \right|^2  \right)^{1/2} .
$$
As $\| n^{-1/2} D^{3/2} \vdg(T)\| = 
(n^{-1} \sum_{i=1}^N d_{i}^3 t_{ii}^2)^{1/2} \leq K$, 
Theorem \ref{theo:quadra}-\eqref{quadra-sum-bounded} can be applied, and
the first term at the r.h.s.~of \eqref{eq:chi2} is of order 
$n^{-1/2}$. We now deal with the second term at the r.h.s.
\[
\E\left| a_j^* (\E_j Q_j) D^{3/2} \frac{\vdg (Q-T)}{\sqrt{n}} \right|
\le K \, \E \left\| \frac{\vdg (Q-T)}{\sqrt{n}} \right\| 
\le \left(\frac Kn \sum_{i=1}^N \E |q_{ii} - t_{ii}|^2 
\right)^{1/2} \leq \frac{K}{\sqrt{n}} 
\] 
by \eqref{eq:estimate-t-tilde}. We now consider the third term. Since 
$\| a_j^* (\E_j Q_j) D^{3/2} \|$ is uniformly bounded, 
\begin{multline*} 
\frac 1n \sum_{j=1}^n 
\E\left| a_j^* (\E_j Q_j) D^{3/2} \frac{\vdg (Q_j-Q)}{\sqrt{n}} \right| \\ 
= 
\frac{1}{n^{3/2}} \sum_{j=1}^n 
\E \left| \tr \left( \diag( a_j^* (\E_j Q_j) D^{3/2} ) (Q_j - Q) \right) 
\right|  
\leq \frac{K}{\sqrt{n}} 
\end{multline*} 
by Lemma \ref{lem:rank1-perturbation}. 
\end{proof}

\begin{lemma}
\label{lemma:chi3} 
Assume that the setting of Theorem \ref{th:CLT} holds true, then:
\[
\sum_{j=1}^n \chi_{3j} + 
\log \left(
\left( 1- \frac 1n \mathrm{Tr}\, D^{\frac 12} T A(I + \delta \tilde D)^{-2}
\tilde D A^* T D^{\frac 12} \right)^2 - \rho^2 \gamma \tilde \gamma \right)
\cvgP{n\to\infty} 0 . 
\]
\end{lemma}

\begin{lemma}
\label{lemma:chi4} 
Assume that the setting of Theorem \ref{th:CLT} holds true, then:
\begin{eqnarray*}
\sum_{j=1}^n \chi_{4j} + 
\log \left( 
\left| 1- 
\vartheta \frac 1n \mathrm{Tr}\, D^{\frac 12} \bar{T} \bar{A}
(I +\delta \tilde D)^{-2} \tilde D A^* T D^{\frac 12} \right|^2
-|\vartheta|^2\rho^2 \underline{\gamma}\, \underline{\tilde \gamma} 
\right) 
\cvgP{n\to\infty} 0\ .
\end{eqnarray*}
\end{lemma}

The core of the paper is devoted to the proof of Lemma \ref{lemma:chi3}. 
This proof is provided in Section \ref{sec:proof-clt-2}. The proof of 
Lemma \ref{lemma:chi4} follows the same canvas with minor differences. 
Elements of this proof are given in Section \ref{sec:proof-clt-3}. 

\section{Proof of Theorem \ref{th:CLT} (part II)} 
\label{sec:proof-clt-2}
This section is devoted to the proof of Lemma \ref{lemma:chi3}.
We begin with the following lemma which implies that 
$\sum_{j=1}^n \chi_{3j}$ can be replaced by its expectation.

\begin{lemma}
\label{lem:control_variances}
For any $N \times 1$ vector $a$ with bounded Euclidean norm, we have,
$$
\max_j \var (a^*(\E_j Q)D(\E_j Q) a)\ =\ {\mathcal O}(n^{-1}) \quad \textrm{and}\quad 
\max_j \var\left(\tr \, (\E_j Q) D (\E_j Q) D \right) \ = \ {\mathcal O}(1) . 
$$
\end{lemma}
Proof of Lemma \ref{lem:control_variances} is postponed to Appendix 
\ref{app:control_variances}. Observe that: 
\begin{align*}
\sum_{i=1}^j \chi_{3j} &= 
\frac{1}{n} \sum_{j=1}^n \rho^2 \tilde t_{jj}^2 
\left( \frac{\tilde d_j^2}{n}  \tr (\E_j Q_j) D (\E_j Q_j) D + 
2 \tilde d_j a_j^* (\E_j Q_j) D (\E_j Q_j) a_j \right) \\ 
&= 
\frac{1}{n} \sum_{j=1}^n \rho^2 \tilde t_{jj}^2 
\left( \frac{\tilde d_j^2}{n}  \tr (\E_j Q) D (\E_j Q) D + 
2 \tilde d_j a_j^* (\E_j Q_j) D (\E_j Q_j) a_j \right) 
+ {\mathcal O}(n^{-1}) \ , 
\end{align*} 
due to Lemma \ref{lem:rank1-perturbation}. Consider the following notations: 
\begin{eqnarray*}
\psi_j &=& \frac 1n \tr \E\left[ (\E_j Q) D (\E_j Q) D \right] = 
\frac 1n \tr \E\left[ (\E_j Q) D Q D \right]\ ,  \\ 
\zeta_{kj} &=& \E\left[ a_k^* (\E_j Q) D (\E_j Q) a_k \right] =  
\E\left[ a_k^* (\E_j Q) D Q a_k \right]\ ,   \\
\theta_{kj} &=& \E\left[ a_k^* (\E_j Q_k) D (\E_j Q_k) a_k \right] =  
\E\left[ a_k^* (\E_j Q_k) D Q_k a_k \right]\ ,   \\
\varphi_j &=& 
\frac 1n \sum_{k=1}^j \rho^2 \tilde d_k \tilde t_{kk}^2 \theta_{kj} \ . 
\end{eqnarray*} 
Thanks to Lemma \ref{lem:control_variances}, we only need to show that 
\begin{equation}
\frac 1n \sum_{j=1}^n \left( \rho^2 \tilde t_{jj}^2 \tilde d_j^2 \psi_j 
+ 2 \rho^2 \tilde d_j \tilde t_{jj}^2  \theta_{jj} \right)
\quad + \quad \log \Delta_n \quad \xrightarrow[n\to\infty]{} 0\ . 
\label{eq-cvg-mean-chi3} 
\end{equation} 

There are structural links between the various
quantities $\psi_j$, $\zeta_{kj}$, $\theta_{kj}$ and $\varphi_j$. The
idea behind the proof is to establish the equations between these
quantities. Solving these equations will yield explicit expressions
which will enable to identify $\frac 1n \sum_{j=1}^n \left( \rho^2
  \tilde t_{jj}^2 \tilde d_j^2 \psi_j + 2 \rho^2 \tilde d_j \tilde
  t_{jj}^2 \theta_{jj} \right)$ as the deterministic quantity 
$-\log \Delta_n$ up to a vanishing error term. 

Proof of \eqref{eq-cvg-mean-chi3} is broken down into four steps. In
the first step, we establish an equation between $\zeta_{kj}$,
$\psi_j$ and $\varphi_j$ (up to ${\mathcal O}(n^{-1/2})$):
Eq. \eqref{eq:zeta=f(psi,phi)}. In the second step, we establish an
equation between $\psi_j$ and $\varphi_j$:
Eq. \eqref{eq:psi-final}. In the third step, we establish an equation
between $\zeta_{kj}$, $\psi_j$ and $\theta_{kj}$:
Eq. \eqref{eq:zeta=f(psi,theta)}. Gathering these results, we obtain a
$2\times 2$ linear system \eqref{eq:system-2x2} whose solutions are
$\psi_j$ and $\varphi_j$. In the fourth step, we solve this system and
finally establish \eqref{eq-cvg-mean-chi3}.

\subsection{Step 1: Expression of 
$\zeta_{kj} = \E [a_k^* (\E_j Q) D Q a_k]$} Writing 
\begin{equation} 
\label{eq-resolvent-identity} 
Q = T + T ( T^{-1} - Q^{-1} ) Q = 
T + T\left( \rho \tilde \delta D + A (I + \delta \tilde D)^{-1} A^* 
- \Sigma\Sigma^* \right) Q\ ,
\end{equation} 
we have: 
\begin{align} 
\zeta_{kj} &= \E \left[ a_k^* 
\E_j\left[ T + T\left( \rho \tilde \delta D + A (I + \delta \tilde D)^{-1} A^* 
- \Sigma\Sigma^* \right) Q \right] D Q a_k \right]\ ,  \nonumber \\ 
&= 
\E [ a_k^* T D Q a_k ]  
+ \rho \tilde\delta \, \E[ a_k^* T D (\E_j Q) D Q a_k ] \nonumber \\
&\phantom{=}  
+ \E[ a_k^* TA (I + \delta \tilde D)^{-1} A^* (\E_j Q) D Q a_k ] 
- \E[ a_k^* T (\E_j \Sigma\Sigma^* Q ) D Q a_k ] \ , \label{eq:zeta1} \\
&\stackrel{\triangle}= 
a_k^* T D T a_k + \rho \tilde\delta \, \E[ a_k^* T D (\E_j Q) D Q a_k ] 
+ X + Z + \varepsilon\ , 
\label{eq:zetakj} 
\end{align} 
where $X$ and $Z$ are the last two terms at the r.h.s.~of \eqref{eq:zeta1}
and where $|\varepsilon| ={\mathcal O}(n^{-1/2})$ by Theorem 
\ref{theo:quadra}-\eqref{conv-quadra}. Beginning with $X$, we have 
\begin{align*} 
X &= 
\sum_{\ell=1}^n \frac{\E[ a_k^* T a_\ell a_\ell^* (\E_j Q) D Q a_k]}  
{1+ \delta \tilde d_\ell} \\
&= 
\sum_{\ell=1}^n \frac{\E[ a_k^* T a_\ell a_\ell^* (\E_j Q_\ell) D Q a_k]}  
{1+ \delta \tilde d_\ell} 
- \sum_{\ell=1}^n \frac{\E[ \rho \tilde t_{\ell\ell}
a_k^* T a_\ell a_\ell^* (\E_j Q_\ell \eta_\ell \eta_\ell^* Q_\ell) D Q a_k]}  
{1+ \delta \tilde d_\ell} 
+ \varepsilon_1 \ ,\\
&= 
\sum_{\ell=1}^n \frac{\E[ a_k^* T a_\ell a_\ell^* (\E_j Q_\ell) D Q a_k]}  
{1+ \delta \tilde d_\ell} 
- \sum_{\ell=1}^n \frac{\E[ \rho \tilde t_{\ell\ell}
a_k^* T a_\ell a_\ell^* (\E_j Q_\ell a_\ell \eta_\ell^* Q_\ell) D Q a_k]}  
{1+ \delta \tilde d_\ell} 
+ \varepsilon_1 + \varepsilon_2\ , \\
&= \sum_{\ell=1}^n \frac{\E[ a_k^* T a_\ell a_\ell^* (\E_j Q_\ell) D Q a_k]}  
{1+ \delta \tilde d_\ell} 
- \sum_{\ell=1}^n \frac{\E[ \rho \tilde t_{\ell\ell}
a_k^* T a_\ell a_\ell^* {\mathcal T}_\ell a_\ell 
(\E_j \eta_\ell^* Q_\ell) D Q a_k]}{1+ \delta \tilde d_\ell} 
+ \varepsilon_1 + \varepsilon_2 + \varepsilon_3\ , \\
&\stackrel{\triangle}= X_1 + X_2 + \varepsilon_1 + \varepsilon_2 + \varepsilon_3 \ ,
\end{align*} 
where 
\begin{eqnarray}
\varepsilon_1 &=& 
- \sum_{\ell=1}^n \frac{\E[ 
a_k^* T a_\ell (\E_j \, ( \rho \tilde q_{\ell\ell} - \rho \tilde t_{\ell\ell} )
a_\ell^* Q_\ell \eta_\ell \eta_\ell^* Q_\ell) D Q a_k]}  
{1+ \delta \tilde d_\ell} \ , \label{eq:epsilon_1} \\  
\varepsilon_2 &=& 
- \sum_{\ell=1}^n \frac{\E[ \rho \tilde t_{\ell\ell}
a_k^* T a_\ell a_\ell^* (\E_j Q_\ell y_\ell \eta_\ell^* Q_\ell) D Q a_k]}  
{1+ \delta \tilde d_\ell}\ ,  \nonumber \\
\varepsilon_3 &=& 
- \sum_{\ell=1}^n \frac{\E[ \rho \tilde t_{\ell\ell}
a_k^* T a_\ell (\E_j a_\ell ( Q_\ell - {\mathcal T}_\ell) a_\ell \, 
\eta_\ell^* Q_\ell) D Q a_k]}{1+ \delta \tilde d_\ell} \ . \nonumber 
\end{eqnarray} 
Using \eqref{eq:diag-coresolvent} and \eqref{eq:woodbury2}, $\varepsilon_1$
can be written as: 
\[
\varepsilon_1 = \E \left[ \E_j \left( a_k^* T A \, 
\diag(\xi_\ell) \, (I+\delta \tilde D)^{-1} \Sigma^* Q \right) 
D Q a_k \right]\ , 
\]
where $\xi_\ell = \rho ( \tilde q_{\ell\ell} - \tilde t_{\ell\ell} )
(1 + \eta_\ell^* Q_\ell \eta_\ell) a_\ell^* Q_\ell
\eta_\ell$. Recalling that $\| \Sigma^* Q \|$ is bounded, we obtain
$|\varepsilon_1 | \leq K \E \| a_k^* T A \diag(\xi_\ell) \| \leq K (
\sum_{\ell=1}^n |[ a_k^* T A ]_\ell|^2 \E \xi_\ell^2 )^{1/2} \leq K /
\sqrt{n}$ by \eqref{eq:estimate-t-tilde} and the boundedness of 
$\E|X_{11}|^{16}$ (Assumption {\bf A\ref{ass:X}}). We show similarly that
$\varepsilon_2$ and $\varepsilon_3$ (with the help of Theorem
\ref{theo:quadra}-\eqref{eq-u(Qj-Tj)v}) are of order ${\mathcal
  O}(n^{-1/2})$. We now develop $X_2$ as:
\begin{align*} 
X_2 &=  
- \sum_{\ell=1}^n \frac{\E[ \rho \tilde t_{\ell\ell} \, 
a_k^* T a_\ell \, a_\ell^* {\mathcal T}_\ell a_\ell \, 
a_\ell^* (\E_j Q_\ell) D Q a_k]}{1+ \delta \tilde d_\ell} 
- \sum_{\ell=1}^j \frac{\E[ \rho \tilde t_{\ell\ell} \, 
a_k^* T a_\ell \, a_\ell^* {\mathcal T}_\ell a_\ell \, 
y_\ell^* (\E_j Q_\ell) D Q a_k]}{1+ \delta \tilde d_\ell} \ ,\\ 
&\stackrel{\triangle}= U_1 + U_2 \ . 
\end{align*}
The term $U_2$ can be expressed as: 
\begin{align*}
U_2 &= 
\sum_{\ell=1}^j \frac{\E[ \rho^2 \tilde t_{\ell\ell}\tilde q_{\ell\ell} \, 
a_k^* T a_\ell \, a_\ell^* {\mathcal T}_\ell a_\ell \, 
y_\ell^* (\E_j Q_\ell) D Q_\ell \eta_\ell \, \eta_\ell^* Q_\ell a_k]}
{1+ \delta \tilde d_\ell}\ , \\
&= 
\sum_{\ell=1}^j \frac{\E[ \rho^2 \tilde t^2_{\ell\ell} \, 
a_k^* T a_\ell \, a_\ell^* {\mathcal T}_\ell a_\ell \, 
y_\ell^* (\E_j Q_\ell) D Q_\ell \eta_\ell \, \eta_\ell^* Q_\ell a_k]}
{1+ \delta \tilde d_\ell} + {\mathcal O}(n^{-1/2}) \ . 
\end{align*} 
Write $\eta_\ell \eta_\ell^* 
= a_\ell a_\ell^* + a_\ell y_\ell^* + y_\ell y_\ell^* + y_\ell a_\ell^*$. 
The term in $a_\ell a_\ell^*$ is zero. Turning to the term in 
$a_\ell y_\ell^*$, we have 
$\E| y_\ell^* (\E_j Q_\ell) D Q_\ell a_\ell \, y_\ell^* Q_\ell a_k| 
= {\mathcal O}(n^{-1})$, hence 
\[
\sum_{\ell=1}^j \frac{\left| \E[ \rho^2 \tilde t^2_{\ell\ell} \, 
a_k^* T a_\ell \, a_\ell^* {\mathcal T}_\ell a_\ell \,  
y_\ell^* (\E_j Q_\ell) D Q_\ell a_\ell \, y_\ell^* Q_\ell a_k] \right|}
{1+ \delta \tilde d_\ell}  
\ \leq\  
\frac Kn \sum_{\ell=1}^j | a_k^* T a_\ell |\ \leq\ \frac{K}{\sqrt{n}} . 
\]
Moreover, 
\begin{eqnarray*} 
\lefteqn{\sum_{\ell=1}^j \frac{\left|\E[ \rho^2 \tilde t^2_{\ell\ell} \, 
a_k^* T a_\ell \, a_\ell^* {\mathcal T}_\ell a_\ell \, 
y_\ell^* (\E_j Q_\ell) D Q_\ell y_\ell \, y_\ell^* Q_\ell a_k]\right|}
{1+ \delta \tilde d_\ell}}  \\ 
&=& 
\sum_{\ell=1}^j \frac{\left|\E\left[ \rho^2 \tilde t^2_{\ell\ell} \, 
a_k^* T a_\ell \, a_\ell^* {\mathcal T}_\ell a_\ell \, 
\left( y_\ell^* (\E_j Q_\ell) D Q_\ell y_\ell - 
\tilde d_\ell n^{-1} \tr D (\E_j Q_\ell) D Q_\ell \right) 
y_\ell^* Q_\ell a_k\right]\right|}
{1+ \delta \tilde d_\ell}  \\ 
&\leq& \frac Kn \sum_{\ell=1}^j | a_k^* T a_\ell | \quad=\quad  {\mathcal O}(n^{-1/2})\ .
\end{eqnarray*} 
The term in $y_\ell a_\ell^*$ is written as 
\[
\sum_{\ell=1}^j \frac{\E[ \rho^2 \tilde t^2_{\ell\ell} \, 
a_k^* T a_\ell \, a_\ell^* {\mathcal T}_\ell a_\ell \, 
y_\ell^* (\E_j Q_\ell) D Q_\ell y_\ell \, a_\ell^* Q_\ell a_k]}
{1+ \delta \tilde d_\ell}
= 
\psi_j 
\sum_{\ell=1}^j \frac{\E[ \rho^2 \tilde d_\ell \tilde t^2_{\ell\ell} \, 
a_k^* T a_\ell \, a_\ell^* {\mathcal T}_\ell a_\ell \, a_\ell^* Q_\ell a_k]}
{1+ \delta \tilde d_\ell} + \varepsilon 
\]
where $\varepsilon ={\mathcal O}(n^{-1})$ by Lemmas \ref{lem:control_variances}
and \ref{lem:rank1-perturbation}. The remaining term in the r.h.s. can be 
handled by the following lemma which is proven in appendix 
\ref{anx-lemma-E(aQa)}: 
\begin{lemma}
\label{lm-E(aQa)}
Let $(u)=(u_n)_{n\in \mathbb{N}}$ be a sequence of vectors with bounded 
Euclidean norms. Let $(\alpha_\ell)_{1\le \ell\le n} = 
(\alpha_{\ell,n})_{1\le \ell\le n}$ be an array of bounded real numbers. 
Then:
\[
\sum_{\ell=1}^j \alpha_\ell \, u^* T a_\ell \, \E\left[ a_\ell Q_\ell u \right]
= 
\sum_{\ell=1}^j 
\frac{\alpha_\ell \, u^* T a_\ell \, a_\ell^* T u }
{\rho\tilde t_{\ell\ell} (1 + \tilde d_\ell \delta)} 
+ {\mathcal O}(n^{-1/2}) \ . 
\]
\end{lemma} 
Applying this lemma with 
$u= a_k$ and $\alpha_\ell = \rho^2 \tilde t_{\ell\ell}^2 \tilde d_\ell 
( 1 + \tilde d_\ell )^{-1} a_\ell {\mathcal T}_\ell a_\ell$, we obtain 
\[
U_2 = 
\psi_j 
\sum_{\ell=1}^j \frac{\rho \tilde d_\ell \tilde t_{\ell\ell} \, 
a_k^* T a_\ell \, a_\ell^* {\mathcal T}_\ell a_\ell \, a_\ell^* T a_k}
{(1+ \delta \tilde d_\ell)^2} + {\mathcal O}(n^{-1/2}) . 
\]
Gathering these results, and using the identity 
$(1 - \rho\tilde t_{\ell\ell} a_\ell^* {\mathcal T}_{\ell} a_\ell)
= \rho\tilde t_{\ell\ell} (1 + \tilde d_\ell \delta)$ 
(see \eqref{eq:t-tilde-jj}), we obtain 
\begin{equation}
X = \sum_{\ell=1}^n \rho \tilde t_{\ell\ell} \, a_k^* T a_\ell \, 
\E[ a_\ell^* (\E_j Q_\ell) D Q a_k ] 
+ \psi_j \sum_{\ell=1}^j \frac{\rho \tilde d_\ell \tilde t_{\ell\ell} \, 
a_k^* T a_\ell \, a_\ell^* {\mathcal T}_{\ell} a_\ell \, 
a_\ell^* T a_k}{(1 + \tilde d_\ell \delta)^2} 
+ {\mathcal O}(n^{-1/2}). 
\label{eq:X-in-zeta} 
\end{equation} 
We now turn to the term $Z$ in \eqref{eq:zetakj}. 
\[ 
Z = 
- \sum_{\ell=1}^n \E [ a_k^* T 
(\E_j \rho \tilde q_{\ell\ell} \eta_\ell\eta_\ell^* Q_\ell ) D Q a_k ]  
= 
- \sum_{\ell=1}^n \rho \tilde t_{\ell\ell} 
\E [ a_k^* T (\E_j \eta_\ell\eta_\ell^* Q_\ell ) D Q a_k ] 
+ \varepsilon \ ,
\]
where 
$$
\varepsilon = \sum_{\ell=1}^n \E[ a_k^* T \left( \E_j \,  
\rho (\tilde{q}_{\ell\ell} - \tilde{t}_{\ell\ell}) 
 \eta_\ell \eta_\ell^* Q_\ell \right) D Q a_k ]
$$ satisfies $\varepsilon = {\mathcal O}(n^{-1/2})$ 
(same arguments as for $\varepsilon_1$ in 
\eqref{eq:epsilon_1}). 
Writing $\eta_\ell \eta_\ell^* = a_\ell a_\ell^* + y_\ell y_\ell^* 
+ a_\ell y_\ell^* + y_\ell a_\ell^*$, we obtain:
\begin{align*} 
Z &= 
- \sum_{\ell=1}^n \rho \tilde t_{\ell\ell} \, a_k^* T a_\ell \, 
\E [a_\ell^* (\E_j Q_\ell ) D Q a_k ] \\
&\phantom{=} - \left( 
\sum_{\ell=1}^j \rho \tilde t_{\ell\ell} \, 
\E [ a_k^* T y_\ell \, y_\ell^* (\E_j Q_\ell ) D Q a_k ]  
+ \frac 1n \sum_{\ell=j+1}^n \rho \tilde t_{\ell\ell} \tilde d_\ell \, 
\E [ a_k^* T D (\E_j Q_\ell ) D Q a_k ] \right) \\ 
&\phantom{=} - \sum_{\ell=1}^j \rho \tilde t_{\ell\ell} \, a_k^* T a_\ell \, 
\E [ y_\ell^* (\E_j Q_\ell ) D Q a_k ]  
- \sum_{\ell=1}^j \rho \tilde t_{\ell\ell} \, 
\E [ a_k^* T y_\ell \, a_\ell^* (\E_j Q_\ell ) D Q a_k ] + 
{\mathcal O}(n^{-1/2}) \\
&\stackrel{\triangle}= Z_1 + Z_2 + Z_3 + Z_4 + {\mathcal O}(n^{-1/2}) . 
\end{align*} 
The term $Z_1$ cancels with the first term in the decomposition of $X$
(first term at the r.h.s.~of \eqref{eq:X-in-zeta}). The term $Z_2$
can be written as: 
\begin{eqnarray*} 
Z_2 &=& 
- \left( 
\sum_{\ell=1}^j \rho \tilde t_{\ell\ell} \, 
\E [ a_k^* T y_\ell \, y_\ell^* (\E_j Q_\ell ) D Q_\ell a_k ] 
+ \frac 1n 
\sum_{\ell=j+1}^n \rho \tilde t_{\ell\ell} \tilde d_\ell \, 
\E [ a_k^* T D (\E_j Q_\ell ) D Q a_k ] \right) \\ 
&& + \sum_{\ell=1}^j \rho^2 \tilde t_{\ell\ell}^2 \, 
\E [ a_k^* T y_\ell \, y_\ell^* 
(\E_j Q_\ell) D Q_\ell\eta_\ell \, \eta_\ell^* Q_\ell a_k ] + \varepsilon \\ 
&\stackrel{\triangle}=& W_1 + W_2 + \varepsilon \ ,
\end{eqnarray*} 
where $\varepsilon$ follows from the substitution of $\rho\tilde q_{\ell\ell}$
with $\rho\tilde t_{\ell\ell}$ and satisfies 
$\varepsilon = {\mathcal O}(n^{-1/2})$ as in \eqref{eq:epsilon_1}. 
Consider first $W_1$:
\[
W_1 = - 
\left( 
\frac 1n \sum_{\ell=1}^j \rho \tilde t_{\ell\ell} \tilde d_\ell 
\E [ a_k^* T D (\E_j Q_\ell ) D Q_\ell a_k ] 
+ 
\frac 1n \sum_{\ell=j+1}^n \rho \tilde t_{\ell\ell} \tilde d_\ell 
\E [ a_k^* T D (\E_j Q_\ell ) D Q a_k ] 
\right) . 
\]
Write: 
$$
(\E_j Q_\ell ) D Q_\ell - (\E_j Q) D Q = (\E_j Q_\ell) D (Q_\ell - Q) + 
(\E_j Q_\ell - \E_j Q) D Q\ .
$$ 
Using \eqref{eq:Qj=f(Q,eta)} and \eqref{eq:1+etaQeta},
\begin{eqnarray*} 
\frac 1n \sum_{\ell=1}^j \rho \tilde t_{\ell\ell} \tilde d_\ell 
\left| \E [ a_k^* T D (\E_j Q_\ell ) D (Q_\ell - Q) a_k ] \right|
&\leq& \frac Kn \sum_{\ell=1}^j 
\left( \E | (1 + \eta_\ell^* Q_\ell \eta_\ell) |^2 \right)^{1/2} 
\left( \E | \eta_\ell^* Q a_k |^2 \right)^{1/2} \\
&\leq& \frac{K}{\sqrt{n}} (\E a_k^* Q \Sigma_{1:j} \Sigma_{1:j}^* Q a_k)^{1/2}
\ =\ {\mathcal O}(n^{-1/2})\ ,
\end{eqnarray*} 
and the same arguments apply to the term $(\E_j Q_\ell - \E_j Q) D Q$. 
Hence,
\[
W_1 = 
- \rho \tilde\delta \, \E [ a_k^* T D (\E_j Q ) D Q a_k ] 
+ {\mathcal O}(n^{-1/2}) . 
\]
Turning to $W_2$, we have: 
\begin{multline*} 
\sum_{\ell=1}^j \rho^2 \tilde t_{\ell\ell}^2 \, 
\E [ a_k^* T y_\ell \, y_\ell^* (\E_j Q_\ell) D Q_\ell y_\ell \, 
a_\ell^* Q_\ell a_k ] \\ 
= 
\sum_{\ell=1}^j \rho^2 \tilde t_{\ell\ell}^2 \, 
\E \Bigl[ a_k^* T y_\ell \, 
\Bigl( y_\ell^* (\E_j Q_\ell) D Q_\ell y_\ell - 
\frac{\tilde d_\ell}{n} \tr D (\E_j Q_\ell) D Q_\ell \Bigr) 
a_\ell^* Q_\ell a_k \Bigr]  
\end{multline*} 
whose modulus is of order ${\mathcal O}(n^{-1/2})$. The term 
$$
\sum_{\ell=1}^j \rho^2 \tilde t_{\ell\ell}^2 
\E [ a_k^* T y_\ell y_\ell^* 
(\E_j Q_\ell) D Q_\ell a_\ell y_\ell^* Q_\ell a_k ]$$ 
can be handled similarly. 

The term 
\[
\sum_{\ell=1}^j \rho^2 \tilde t_{\ell\ell}^2 
\E [ a_k^* T y_\ell y_\ell^* 
(\E_j Q_\ell) D Q_\ell a_\ell a_\ell^* Q_\ell a_k ] 
= \frac 1n \sum_{\ell=1}^j \rho^2 \tilde t_{\ell\ell}^2 \tilde d_\ell 
\E [ a_k^* T D  (\E_j Q_\ell) D Q_\ell a_\ell a_\ell^* Q_\ell a_k ] 
\]
is bounded by 
$K n^{-1/2}$. Finally,
\begin{eqnarray*}
W_2 &=& 
\sum_{\ell=1}^j \rho^2 \tilde t_{\ell\ell}^2 
\E [ y_\ell^* Q_\ell a_k a_k^* T y_\ell \, 
y_\ell^* (\E_j Q_\ell) D Q_\ell y_\ell ] + {\mathcal O}(n^{-1/2}) \ , \\
&\stackrel{(a)}=& \psi_j \, a_k^* T D T a_k \, 
\frac 1n \sum_{\ell=1}^j \rho^2 \tilde t_{\ell\ell}^2 
\tilde d_\ell^2 + {\mathcal O}(n^{-1/2}) \ ,
\end{eqnarray*}
where $(a)$ follows by standard arguments as those already developed.

The term $Z_3$ satisfies 
\[
Z_3 = 
\sum_{\ell=1}^j \rho^2 \tilde t_{\ell\ell}^2 a_k^* T a_\ell 
\E [ y_\ell^* (\E_j Q_\ell ) D Q_\ell \eta_\ell \, \eta_\ell^* Q_\ell a_k ]  
+ {\mathcal O}(n^{-1/2}) . 
\]
Writing $\eta_\ell \eta_\ell^* = y_\ell y_\ell^* + a_\ell y_\ell^* +
a_\ell a_\ell^* + y_\ell a_\ell^*$ and relying arguments as those
already developed, one can check that the only non-negligible
contribution stems from the term containing $y_\ell a_\ell^*$. Hence,
\begin{eqnarray*}
Z_3 &=& 
\psi_j \sum_{\ell=1}^j \rho^2  \tilde t_{\ell\ell}^2 \tilde d_\ell
a_k^* T a_\ell \E [ a_\ell^* Q_\ell a_k ] + {\mathcal O}(n^{-1/2}) \ , \\
&=& 
\psi_j \sum_{\ell=1}^j \frac{\rho  \tilde t_{\ell\ell} \tilde d_\ell
a_k^* T a_\ell a_\ell^* T a_k }
{1 + \tilde d_\ell \delta} + {\mathcal O}(n^{-1/2}) \ ,
\end{eqnarray*}
by Lemma \eqref{lm-E(aQa)}. Similarly, 
\begin{eqnarray*}
Z_4&=&\sum_{\ell=1}^j \rho^2 \tilde t_{\ell\ell}^2  \, 
\E [ a_k^* T y_\ell \, 
a_\ell^* (\E_j Q_\ell ) D Q_\ell \eta_\ell \, 
\eta_\ell^* Q_\ell a_k ] + {\mathcal O}(n^{-1/2})\ ,\\
&=& a_k^* TDT a_k 
\frac 1n \sum_{\ell=1}^j \rho^2 \tilde d_\ell \tilde t_{\ell\ell}^2 
\E[ a_\ell^* (\E_j Q_\ell) D Q_\ell a_\ell ] + {\mathcal O}(n^{-1/2}) \ ,\\
&=& a_k^* TDT a_k \varphi_j + {\mathcal O}(n^{-1/2}) \ .
\end{eqnarray*}


Gathering these results, we obtain
\begin{multline*}
Z = 
- \sum_{\ell=1}^n \rho \tilde t_{\ell\ell} \, a_k^* T a_\ell \, 
\E [a_\ell^* (\E_j Q_\ell ) D Q a_k ] 
- \rho \tilde\delta \, \E [ a_k^* T D (\E_j Q ) D Q a_k ] \\ 
+ \psi_j \, a_k^* T D T a_k \, 
\frac 1n \sum_{\ell=1}^j \rho^2 \tilde t_{\ell\ell}^2 
\tilde d_\ell^2  
+ \psi_j \sum_{\ell=1}^j \frac{\rho  \tilde t_{\ell\ell} \tilde d_\ell
a_k^* T a_\ell a_\ell^* T a_k }
{1 + \tilde d_\ell \delta} 
+ a_k^* TDT a_k \, \varphi_j + {\mathcal O}(n^{-1/2}) . 
\end{multline*} 
Plugging this and Eq. \eqref{eq:X-in-zeta} into \eqref{eq:zetakj}, and 
noticing that 
$\rho\tilde t_{\ell\ell} ( a_\ell {\mathcal T}_\ell a_\ell 
(1 + \tilde d_\ell \delta)^{-1} + 1) = (1 + \tilde d_\ell \delta)^{-1}$, 
we obtain: 
\begin{multline}
\zeta_{kj} =  
a_k^* T D T a_k + \psi_j \left( 
\sum_{\ell=1}^j 
\frac{ a_k^* T a_\ell \, \tilde d_\ell \, a_\ell^* T a_k}
{(1 + \tilde d_\ell \delta)^2}  \ + \ 
a_k^* T D T a_k \, 
\frac 1n \sum_{\ell=1}^j \rho^2 \tilde t_{\ell\ell}^2 \tilde d_\ell^2  
\right) \\ 
+ a_k^* TDT a_k \, \varphi_j + {\mathcal O}(n^{-1/2}) \ . 
\label{eq:zeta=f(psi,phi)} 
\end{multline} 

\subsection{Step 2: Expression of 
$\psi_j = n^{-1} \tr \E [ (\E_j Q) D Q D ]$} 
Using Identity \eqref{eq-resolvent-identity}, we obtain:
\begin{align} 
\psi_j &= 
\frac{1}{n} \tr \E[ T DQD ] 
+ \frac{\rho \tilde\delta}{n} \tr \E[ TD (\E_j Q) D Q D ] \nonumber \\ 
&\phantom{=} 
+ \frac 1n \tr \E [ T A (I + \delta \tilde D)^{-1} A^* (\E_j Q) D Q D ] 
- \frac 1n \tr \E [ T (\E_j \Sigma\Sigma^* Q ) D Q D ] \ ,
\label{eq-terms-psi} \\
&= 
\frac{1}{n} \tr DTDT  
+ \frac{\rho \tilde\delta}{n} \tr \E[ TD (\E_j Q) D Q D ] 
+ X + Z + \varepsilon \ ,
\label{eq:psi} 
\end{align} 
where $X$ and $Z$ are the last two terms of the r.h.s.~of
\eqref{eq-terms-psi}, and where $\varepsilon ={\mathcal O}(n^{-1})$ by
Theorem \ref{theo:quadra}-\eqref{speed(T-Q)}. Due to the presence of
the multiplying factor $n^{-1}$, the treatment of $X$ and $Z$ is
simpler here than the treatment of their analogues for $\zeta_{kj}$.
We skip hereafter the details related to the bounds over the
$\varepsilon$'s. The term $X$ satisfies
\begin{align*} 
X &= 
\frac 1n \sum_{\ell=1}^n 
\frac{\E[ a_\ell^* (\E_j Q) D Q D  T a_\ell ]}{1 + \delta \tilde d_\ell} \ ,\\
&= 
\frac 1n \sum_{\ell=1}^n 
\frac{\E[ a_\ell^* (\E_j Q_\ell) D Q D  T a_\ell ]}{1 + \delta \tilde d_\ell} 
- 
\frac 1n \sum_{\ell=1}^n 
\frac{\E[ \rho \tilde t_{\ell\ell} \, 
( \E_j a_\ell^* Q_\ell \eta_\ell \, \eta_\ell^* Q_\ell ) D Q D T a_\ell ]}
{1 + \delta \tilde d_\ell} + \varepsilon \ ,\\
&= 
\frac 1n \sum_{\ell=1}^n 
\frac{\E[ a_\ell^* (\E_j Q_\ell) D Q D  T a_\ell ]}{1 + \delta \tilde d_\ell} 
- \frac 1n \sum_{\ell=1}^n \frac{\E[ \rho \tilde t_{\ell\ell} \, 
a_\ell^* {\mathcal T}_\ell a_\ell \, 
a_\ell^* ( \E_j Q_\ell ) D Q D T a_\ell ]}
{1 + \delta \tilde d_\ell} \\ 
&\phantom{=} 
- \frac 1n \sum_{\ell=1}^j \frac{\E[ \rho \tilde t_{\ell\ell} \, 
a_\ell^* {\mathcal T}_\ell a_\ell \, 
y_\ell^* ( \E_j Q_\ell ) D Q D T a_\ell ]}
{1 + \delta \tilde d_\ell} + \varepsilon' \ ,
\end{align*} 
where $\max(|\varepsilon|, |\varepsilon'| ) = {\mathcal O}(n^{-1/2})$. As 
$1 - \rho \tilde t_{\ell\ell} a_\ell^* {\mathcal T}_\ell a_\ell = 
\rho \tilde t_{\ell\ell} (1 + \tilde d_\ell \delta)$, 
\begin{align} 
X &= 
\frac 1n \sum_{\ell=1}^n 
\E[ \rho \tilde t_{\ell\ell} \, 
a_\ell^* (\E_j Q_\ell) D Q D  T a_\ell ] 
- \frac 1n \sum_{\ell=1}^j \frac{\E[ \rho \tilde t_{\ell\ell} \, 
a_\ell^* {\mathcal T}_\ell a_\ell \, 
y_\ell^* ( \E_j Q_\ell ) D Q D T a_\ell ]}
{1 + \delta \tilde d_\ell} + {\mathcal O}(n^{-1/2}) \ ,\nonumber \\ 
&= 
\frac 1n \sum_{\ell=1}^n 
\E[ \rho \tilde t_{\ell\ell} \, 
a_\ell^* (\E_j Q_\ell) D Q D  T a_\ell ] 
+ \frac 1n \sum_{\ell=1}^j \frac{\E[ \rho^2 \tilde t_{\ell\ell}^2 \, 
a_\ell^* {\mathcal T}_\ell a_\ell \, 
y_\ell^* ( \E_j Q_\ell ) D Q_\ell \eta_\ell \, \eta_\ell^* Q_\ell D T a_\ell ]}
{1 + \delta \tilde d_\ell} + {\mathcal O}(n^{-1/2}) \ ,\nonumber \\ 
&= 
\frac 1n \sum_{\ell=1}^n 
\E[ \rho \tilde t_{\ell\ell} \, 
a_\ell^* (\E_j Q_\ell) D Q D  T a_\ell ] 
+ \frac{\psi_j}{n} \sum_{\ell=1}^j 
\frac{\tilde d_\ell \, a_\ell^* T a_\ell \, a_\ell^* T D T a_\ell}
{(1 + \delta \tilde d_\ell)^3} + {\mathcal O}(n^{-1/2})\ , \label{eq:decompo-X}
\end{align} 
where \eqref{eq:T-cal(T)_j} is used to obtain the last equation. The term $Z$
can be expressed as: 
\begin{align*} 
Z &= 
- \frac 1n \sum_{\ell=1}^n 
\tr \E [\rho \tilde t_{\ell\ell} T (\E_j \eta_\ell \eta_\ell^* Q_\ell ) D Q D ] 
+ {\mathcal O}(n^{-1/2})\ , \\
&= 
- \frac 1n \sum_{\ell=1}^n 
\E [ \rho \tilde t_{\ell\ell} a_\ell^* (\E_j Q_\ell ) D Q D T a_\ell] \\
&\phantom{=}  
- \left( 
\frac 1n \sum_{\ell=1}^j 
\E [ \rho \tilde t_{\ell\ell} y_\ell^* (\E_j Q_\ell ) D Q D T y_\ell] 
+ \frac 1n \sum_{\ell=j+1}^n 
\frac 1n 
\tr \E [ \rho \tilde d_\ell \tilde t_{\ell\ell} T D (\E_j Q_\ell ) D Q D ] 
\right) \\
&\phantom{=} 
- \frac 1n \sum_{\ell=1}^j 
\E [ \rho \tilde t_{\ell\ell} y_\ell^* (\E_j Q_\ell ) D Q D T a_\ell ] 
- \frac 1n \sum_{\ell=1}^j 
\E [ \rho \tilde t_{\ell\ell} a_\ell^* (\E_j Q_\ell ) D Q D T y_\ell ] 
+ {\mathcal O}(n^{-1/2}) \ ,\\
&\stackrel{\triangle}= Z_1 + Z_2 + Z_3 + Z_4 + {\mathcal O}(n^{-1/2}) . 
\end{align*}
The term $Z_1$ cancels with the first term in the r.h.s. of $X$'s decomposition \eqref{eq:decompo-X}.
The terms $Z_2$, $Z_3$ and $Z_4$ satisfy: 
\begin{align*}
Z_2 &= 
- \frac{\rho\tilde\delta}{n} \tr \E [ T D (\E_j Q ) D Q D ] 
+ \frac 1n \sum_{\ell=1}^j 
\E [ \rho^2 \tilde t_{\ell\ell}^2 
y_\ell^* (\E_j Q_\ell ) D Q_\ell \eta_\ell \, \eta_\ell^* Q_\ell D T y_\ell] 
+ {\mathcal O}(n^{-1/2}) \\
&= 
- \frac{\rho\tilde\delta}{n} \tr \E [ T D (\E_j Q ) D Q D ] 
+ \psi_j \frac 1n \tr DTDT 
\frac{1}{n} \sum_{\ell=1}^j \rho^2 \tilde d_\ell^2 \tilde t_{\ell\ell}^2 
+ {\mathcal O}(n^{-1/2}) \ , \\
Z_3 &= 
\frac 1n \sum_{\ell=1}^j 
\E [ \rho^2 \tilde t_{\ell\ell}^2 y_\ell^* (\E_j Q_\ell ) D Q_\ell \eta_\ell \,
\eta_\ell^* Q_\ell D T a_\ell ] + {\mathcal O}(n^{-1/2}) \\
&= 
\psi_j \frac 1n \sum_{\ell=1}^j 
\rho \tilde d_\ell \tilde t_{\ell\ell} \, 
\frac{a_\ell^* TDT a_\ell}{1+\tilde d_\ell \delta}  
+ {\mathcal O}(n^{-1/2}) 
\quad \text{(see Th.\ref{theo:quadra}-(\ref{eq-u(Qj-Tj)v})} \ \text{and} 
\ \eqref{eq:T-cal(T)_j} \text{)}, \\ 
Z_4 &= 
\frac 1n \sum_{\ell=1}^j 
\E [ \rho^2 \tilde t_{\ell\ell}^2 
a_\ell^* (\E_j Q_\ell ) D Q_\ell \eta_\ell \, \eta_\ell^* Q_\ell D T y_\ell ] 
+ {\mathcal O}(n^{-1/2}) \\ 
&= 
\frac 1n \tr DTDT 
\frac 1n \sum_{\ell=1}^j \rho^2 \tilde d_\ell \tilde t_{\ell\ell}^2 
\E [ a_\ell^* (\E_j Q_\ell ) D Q_\ell a_\ell ] + {\mathcal O}(n^{-1/2}) . 
\end{align*}
Plugging these terms in \eqref{eq:psi}, we obtain: 
\begin{align} 
\psi_j &= 
\gamma + \psi_j \left( 
\frac{1}{n} \sum_{\ell=1}^j a_\ell^* T D T a_\ell \left( 
\frac{\tilde d_\ell \, a_\ell^* T a_\ell}
{(1 + \delta \tilde d_\ell)^3} 
+ \frac{\rho \tilde d_\ell \tilde t_{\ell\ell}}{1+\tilde d_\ell \delta} \right) 
+ 
\frac{\gamma}{n}  \sum_{\ell=1}^j \rho^2 \tilde d_\ell^2 \tilde t_{\ell\ell}^2 
\right)  
+ \gamma \varphi_j + {\mathcal O}(n^{-1/2}) \nonumber\ , \\
&= 
\gamma + \psi_j \left( 
\frac{1}{n} \sum_{\ell=1}^j \frac{ \tilde d_\ell \, a_\ell^* T D T a_\ell}
{(1 + \delta \tilde d_\ell)^2}  
+ 
\frac{\gamma}{n} \sum_{\ell=1}^j \rho^2 \tilde d_\ell^2 \tilde t_{\ell\ell}^2 
\right)  
+ \gamma \varphi_j + {\mathcal O}(n^{-1/2}) \ ,
\label{eq:psi-final} 
\end{align} 
using \eqref{eq:t-tilde-jj} and \eqref{eq:T-cal(T)_j}.

\subsection{Step 3: Relation between $\zeta_{kj}$ and $\theta_{kj}$ for
$k \leq j$} 

The term $\zeta_{kj}$ can be written as  
\begin{align*} 
\zeta_{kj} &= 
\E [ a_k^* \E_j(Q_k - \rho \tilde{q}_{kk} Q_k \eta_k \eta_k^* Q_k) D 
(Q_k - \rho \tilde{q}_{kk} Q_k \eta_k \eta_k^* Q_k) a_k ] \\
&= 
\theta_{kj}  
- \rho \tilde{t}_{kk} 
\E [ a_k^* \E_j (Q_k \eta_k \eta_k^* Q_k) D Q_k a_k ] 
- \rho \tilde{t}_{kk} 
\ \E[ a_k^* \E_j(Q_k) D Q_k \eta_k \, \eta_k^* Q_k a_k ] \\ 
&\phantom{=} 
+ \rho^2 \tilde{t}_{kk}^2 \,   
\E [ a_k^* \E_j( Q_k \eta_k \eta_k^* Q_k) D 
Q_k \eta_k \, \eta_k^* Q_k a_k ] + {\mathcal O}(n^{-1/2})   \\
&\stackrel{\triangle}= 
\theta_{kj} + X_1 + X_2 + X_3 + {\mathcal O}(n^{-1/2}) \ . 
\end{align*} 
Using similar arguments as those developed previously, we get:
\begin{eqnarray*}
X_1 &=& - \rho \tilde{t}_{kk} 
a_k^* {\mathcal T}_k a_k \, 
\E[ a_k^* \E_j(Q_k) D Q_k a_k ] + {\mathcal O}(n^{-1/2})
\ =\ - \rho \tilde{t}_{kk} 
a_k^* {\mathcal T}_k a_k \, \theta_{kj} + {\mathcal O}(n^{-1/2}),  
\ \\ 
X_2 &=& 
- \rho \tilde{t}_{kk} a_k^* {\mathcal T}_k a_k \ \theta_{kj} + 
{\mathcal O}(n^{-1/2}) \ . 
\end{eqnarray*} 
As $k\leq j$, 
\begin{eqnarray*}
X_3 &=& 
\rho^2 \tilde{t}_{kk}^2 \, \left( a_k^* {\mathcal T}_k a_k\right)^2 \,  
\E[ \eta_k^* \E_j(Q_k) D Q_k \eta_k ] +  {\mathcal O}(n^{-1/2}) \ ,\\ 
&=& \rho^2 \tilde{t}_{kk}^2  \, \left( a_k^* {\mathcal T}_k a_k\right)^2 \ 
\left( \theta_{kj} + \tilde d_k \psi_j\right) +  {\mathcal O}(n^{-1/2}) \ . 
\end{eqnarray*} 
Using \eqref{eq:t-tilde-jj} and \eqref{eq:T-cal(T)_j}, we finally obtain:
\begin{equation}
\label{eq:zeta=f(psi,theta)} 
\zeta_{kj} = 
\rho^2 \tilde t_{kk}^2 ( 1 + \tilde d_k \delta )^2 \, \theta_{kj} + \, 
\tilde d_k \left( \frac{a_k^* T a_k}{1 + \tilde d_k \delta} \right)^2 \, 
\psi_j \, 
+  {\mathcal O}(n^{-1/2}) \ . 
\end{equation}

\subsection{Step 4: A system of perturbed linear equations
in $(\psi_j, \varphi_j)$. Proof of \eqref{eq-cvg-mean-chi3}} 

Combining \eqref{eq:zeta=f(psi,theta)} with \eqref{eq:zeta=f(psi,phi)}, we
obtain
\begin{multline}
\rho^2 \tilde d_k \tilde t_{kk}^2 \theta_{kj} = 
\tilde d_k \frac{a_k^* T D T a_k}{( 1 + \tilde d_k \delta )^2} \\ 
+ \left( 
\sum_{\ell=1}^j 
\frac{ \tilde d_k \, a_k^* T a_\ell \, \tilde d_\ell \, a_\ell^* T a_k}
{(1 + \tilde d_k \delta)^2 (1 + \tilde d_\ell \delta)^2}  \ + \ 
\frac{\tilde d_k \, a_k^* T D T a_k}{( 1 + \tilde d_k \delta )^2} \, 
\frac 1n \sum_{\ell=1}^j \rho^2 \tilde t_{\ell\ell}^2 \tilde d_\ell^2  
- \tilde d_k^2 \frac{(a_k^* T a_k)^2}{(1 + \tilde d_k \delta)^4}  
\right) \psi_j \\ 
+ \frac{\tilde d_k \, a_k^* TDT a_k}{( 1 + \tilde d_k \delta )^2} \, \varphi_j 
+ {\mathcal O}(n^{-1/2}) 
\label{eq:theta_kj} 
\end{multline} 
which implies that 
$\varphi_j =\frac 1n \sum_{k=1}^j \rho^2 \tilde d_k\tilde t_{kk}^2\theta_{kj}$
satisfies  
\[
(1 - F_j) \varphi_j - \left( G_j + F_j M_j \right) \psi_j  
= F_j + {\mathcal O}(n^{-1/2})
\] 
where 
\begin{eqnarray}
F_j &=& \frac 1n \sum_{k=1}^j 
\frac{a_k^* T D T a_k \, \tilde d_k}{( 1 + \tilde d_k \delta )^2}\ ,\nonumber \\
M_j &=& \frac 1n \sum_{\ell=1}^j \rho^2 \tilde t_{\ell\ell}^2 \tilde d_\ell^2\ ,  \label{eq:FMG}  \\ 
G_j &=& \frac 1n \sum_{k=1}^j 
\sum_{\scriptstyle{\ell=1} \atop \scriptstyle{\ell\neq k}}^j 
\frac{ \tilde d_k \tilde d_\ell \left| a_k^* T a_\ell \right|^2}
{(1 + \tilde d_k \delta)^2 (1 + \tilde d_\ell \delta)^2}\ .  
\nonumber 
\end{eqnarray} 
With these new notations, equation \eqref{eq:psi-final} is rewritten 
\[
-\gamma \varphi_j + (1 - F_j - \gamma M_j) \psi_j = \gamma + 
{\mathcal O}(n^{-1/2}) , 
\]
and we end up with a system of two perturbed linear equations in 
$(\varphi_j, \psi_j)$:
\begin{equation}\label{eq:system-2x2}
\left\{
\begin{array}{lcl}
(1 - F_j) \varphi_j - \left( G_j + F_j M_j \right) \psi_j  
&=& F_j + {\mathcal O}(n^{-1/2})\\
-\gamma \varphi_j + (1 - F_j - \gamma M_j) \psi_j &=& \gamma + 
{\mathcal O}(n^{-1/2}) 
\end{array}
\right.\ .
\end{equation}

The determinant of this system is ${\bs\Delta_j} = (1 - F_j)^2 -
\gamma M_j - \gamma G_j$.  The following lemma establishes the link
between the ${\bs\Delta_j}$'s and $\Delta_n$ as defined in Theorem \ref{th:CLT}.
\begin{lemma}
\label{lm-Deltaj}
Recall the definition of $\Delta_n$ :
$$
\Delta_n = 
\left( 1- \frac 1n \tr D^{\frac 12} T A(I + \delta \tilde D)^{-2} 
\tilde D A^* T D^{\frac 12}
\right)^2 - \rho^2 \gamma \tilde \gamma \ .
$$
The determinants ${\bs\Delta}_j$ decrease as $j$ goes from $1$ to $n$; moreover,  
${\bs\Delta}_n$ coincides with $\Delta_n$.
\end{lemma}  

Proof of Lemma \ref{lm-Deltaj} is postponed to Appendix \ref{prf:lm-Deltaj}.

Solving this system of equations and using the lemma in conjunction with 
the fact $\liminf \Delta_n > 0$, established in Lemma \ref{lm:var-bounds}, we obtain:
\[
\begin{bmatrix} \varphi_j \\ \psi_j \end{bmatrix} 
= 
\frac{1}{{\bs \Delta}_j}
\begin{bmatrix} F_j(1-F_j) + \gamma G_j \\ \gamma \end{bmatrix} 
+ {\bs\varepsilon}_j\ , 
\]
where $\|{\bs\varepsilon}_j\| = {\mathcal O}(n^{-1/2})$. Replacing into \eqref{eq:theta_kj}, we obtain 
\begin{align*} 
\frac{2 \rho^2 \tilde d_j \tilde t_{jj}^2 \theta_{jj}}{n} &= 
2 (F_j\! -\! F_{j-1}) + 
(G_j\! -\! G_{j-1} + 2M_j(F_j\! -\! F_{j-1})) \psi_j 
+ 2 (F_j\! -\! F_{j-1}) \varphi_j + {\mathcal O}(n^{-3/2}) \\
&= 2(F_j - F_{j-1}) + 
\frac{\gamma(G_j - G_{j-1}) + 2 \gamma M_j (F_j - F_{j-1})}{{\bs\Delta}_j} 
\\
& 
\ \ \ \ \ \ \ \ \ \ \ \ \ \ \ \ 
\ \ \ \ \ \ \ \ \ \ \ \ \ \ \ \ 
+ \frac{2(F_j - F_{j-1})(F_j(1-F_j)+\gamma G_j)}{{\bs\Delta}_j} + 
{\mathcal O}(n^{-3/2}) 
\end{align*}
which leads to 
\begin{multline*}
\frac 1n \sum_{j=1}^n \left( \rho^2 \tilde t_{jj}^2 \tilde d_j^2 \psi_j 
+ 2 \rho^2 \tilde d_j \tilde t_{jj}^2  \theta_{jj} \right) \\ 
= 
\sum_{j=1}^n 
\frac{2(F_j - F_{j-1}) (1 - F_j) + \gamma(M_j - M_{j-1}) + 
\gamma(G_j - G_{j-1})}{\bs \Delta_j} 
+ {\mathcal O}(n^{-1/2}) . 
\end{multline*} 
On the other hand, $\bs\Delta_{j-1} - \bs\Delta_{j} = 
2(F_j - F_{j-1}) (1 - F_j) + \gamma(M_j - M_{j-1}) + \gamma(G_j - G_{j-1})
+ {\mathcal O}(n^{-2})$, hence, due to Lemma \ref{lm-Deltaj} and 
to $\liminf \Delta_n > 0$, 
\begin{multline*}
\frac 1n \sum_{j=1}^n \left( \rho^2 \tilde t_{jj}^2 \tilde d_j^2 \psi_j 
+ 2 \rho^2 \tilde d_j \tilde t_{jj}^2  \theta_{jj} \right) 
= \sum_{j=1}^n \frac{\bs\Delta_{j-1} - \bs\Delta_{j}}{\bs\Delta_j} + 
{\mathcal O}(n^{-1/2}) \\
= \sum_{j=1}^n \log\left( 1 + 
\frac{\bs\Delta_{j-1} - \bs\Delta_{j}}{\bs\Delta_j}\right)
 + {\mathcal O}(n^{-1/2}) \\
= \sum_{j=1}^n \log \frac{\bs\Delta_{j-1}}{\bs\Delta_j} + 
{\mathcal O}(n^{-1/2}) 
= - \log(\Delta_n) + {\mathcal O}(n^{-1/2}) 
\end{multline*} 
which proves \eqref{eq-cvg-mean-chi3}. Lemma \ref{lemma:chi3} is proven.

\section{Proof of Theorem \ref{th:CLT} (part III)}
\label{sec:proof-clt-3}
In this section, we complete the proof of Theorem \ref{th:CLT}. Proof
of Lemma \ref{lemma:chi4} is very close to the proof of Lemma
\ref{lemma:chi3}; we therefore only provide its main landmarks. We
finally establish the main estimates over $(\Theta_n)$.

\subsection{Elements of proof for Lemma \ref{lemma:chi4}}

Proof of Lemma \ref{lemma:chi4} relies on the following counterpart of Lemma 
\ref{lm-approx-quadra}: 
\begin{lemma}
\label{lm-approx-quadra-non-circular} 
Assume that the setting of Lemma \ref{lm-approx-quadra} holds true;
and let $\mathbb{E} x^2 = \vartheta$. Then for any $p\geq 2$,
\[
\mathbb{E}|\bs{x}^T M \bs{x} - \vartheta \tr M|^p
\le K_p \left( \left( \mathbb{E}|x_1|^4 \tr M M^* \right)^{p/2} 
+ \mathbb{E} |x_1|^{2p} \tr (M M^*)^{p/2}
\right)\ .
\]
\end{lemma} 
\begin{proof}
The result is obtained upon noticing that 
\begin{gather*}
{\bs x}^T M {\bs x} = 
\frac 14 \sum_{k=0}^3 \mathbf{i}^k 
\left( \mathbf{i}^k \bar{\bs x} + {\bs x} \right)^* M 
\left( \mathbf{i}^k \bar{\bs x} + {\bs x} \right) 
\end{gather*} 
and using Lemma \ref{lm-approx-quadra}. 
\end{proof} 
Here are the main steps of the proof. Introducing the notations 
\begin{eqnarray*}
\underline\psi_j &=& \frac 1n \tr \E\left[ (\E_j Q) D \bar Q D \right]\ ,  \\ 
\underline\zeta_{kj} &= &\E\left[ a_k^* (\E_j Q) D \bar Q \bar a_k \right]\ ,\\  
\underline\theta_{kj} &=& \E\left[ a_k^* (\E_j Q_k) D \bar Q_k \bar a_k \right]\ , 
\\
\underline\varphi_j &=& 
\frac 1n \sum_{k=1}^j \rho^2 \tilde d_k \tilde t_{kk}^2 \underline\theta_{kj} \ , 
\end{eqnarray*} 
and adapting Lemma \ref{lem:control_variances}, we only need to prove that:
\[
\frac 1n \sum_{j=1}^n \left( \rho^2 \tilde t_{jj}^2 \tilde d_j^2 
|\vartheta|^2 \underline\psi_j 
+ 2 \rho^2 \tilde d_j \tilde t_{jj}^2  
\re\left( \vartheta \underline\theta_{jj} \right) \right)
\quad + \quad \log \underline\Delta_n \quad \xrightarrow[n\to\infty]{} 0\ . 
\]
Similar derivations as those performed in Steps 1-3 in Section \ref{sec:proof-clt-2} yield the perturbed system:
$$
\left\{
\begin{array}{rcl}
(1 - \vartheta \underline{F}_j) \underline\varphi_j 
- (\bar\vartheta \underline{G}_j + |\vartheta|^2 \underline{F}_j M_j) 
\underline\psi_j &=& \underline{F}_j + {\mathcal O}(n^{-1/2}) \\
- \vartheta \underline\gamma \underline\varphi_j  
+ ( 1 - \bar\vartheta \bar{\underline{F}}_j -  \underline\gamma |\vartheta|^2 
M_j) \underline\psi_j &=& \underline\gamma + {\mathcal O}(n^{-1/2}) 
\end{array}\right. \ ,
$$
where 
\begin{eqnarray*}
\underline{F}_j &=& \frac 1n \sum_{k=1}^j 
\frac{a_k^* T D \bar T \bar a_k \, \tilde d_k}
{( 1 + \tilde d_k \delta )^2} \in \C,  \qquad  
\underline{G}_j \ =\quad  \frac 1n \sum_{k=1}^j 
\sum_{\scriptstyle{\ell=1} \atop \scriptstyle{\ell\neq k}}^j 
\frac{ \tilde d_k \tilde d_\ell \left( a_k^* T a_\ell \right)^2}
{(1 + \tilde d_k \delta)^2 (1 + \tilde d_\ell \delta)^2} \in \R \ ,\\
M_j &=& \frac 1n \sum_{\ell=1}^j \rho^2 \tilde t_{\ell\ell}^2 \tilde d_\ell^2\ .
\end{eqnarray*}
The determinant of this system is: 
$$
\underline{\bs\Delta}_j = \left| 1 - \vartheta \underline{F}_j \right|^2 
- |\vartheta|^2 \underline\gamma \left(M_j + \underline{G}_j \right) .
$$
By \eqref{ineq:trace}, $0\leq \underline\gamma\leq \gamma$; furthermore,
$|\vartheta| \leq 1$, $| \underline{F}_j | \leq F_j$, 
and $| \underline{G}_j | \leq G_j$. As a result, $\underline{\bs\Delta}_j \geq {\bs\Delta}_j$. 
Hence, by Lemma 
\ref{lm-Deltaj}, the perturbation remains of order ${\mathcal O}(n^{-1/2})$ after solving
the system. Performing the same derivations as in Step 4 in Section \ref{sec:proof-clt-2}, 
it can be established that $\underline{\bs\Delta}_n = \underline{\Delta}_n$. We finally end up with:
\begin{eqnarray*} 
\frac 1n \sum_{j=1}^n \left( \rho^2 \tilde t_{jj}^2 \tilde d_j^2 
|\vartheta|^2 \underline\psi_j 
+ 2 \rho^2 \tilde d_j \tilde t_{jj}^2  
\re\left( \vartheta \underline\theta_{jj} \right) \right)
&=& \sum_{j=1}^n \frac{\underline{\bs\Delta}_{j-1} -  \underline{\bs\Delta}_j}
{\underline{\bs\Delta}_j} + {\mathcal O}(n^{-1/2}) \ ,\\ 
&=& - \log(\underline\Delta_n) + {\mathcal O}(n^{-1/2})\ ,
\end{eqnarray*} 
which is the desired result.

\subsection{Estimates over $\Theta_n$} \label{sec:bound-theta}
In order to conclude the proof of Theorem \ref{th:CLT}, it remains to prove
that $0 < \liminf_n \Theta_n \leq \limsup_n \Theta_n < \infty$. 

Consider first the upper bound. By Lemma \ref{lm:var-bounds}, 
$\sup_n (-\log\Delta_n) < \infty$. 
As $\underline\Delta_n \geq \Delta_n$, $\log\underline\Delta_n$ 
is defined and $\sup_n (-\log\underline\Delta_n) < \infty$. By Lemma 
\ref{lm:estimates}, the cumulant term in the expression of $\Theta_n$ is 
bounded, hence $\limsup_n \Theta_n < \infty$. 

We now prove that $\liminf \Theta_n > 0$. To this end, write: 
\begin{align*} 
\Theta_n &= 
\sum_{j=1}^n \left( 
\frac{{\bs \Delta}_{j-1} - {\bs\Delta}_j}{{\bs\Delta}_j} 
+ \frac{\underline{\bs \Delta}_{j-1} - \underline{\bs\Delta}_j}
{\underline{\bs\Delta}_j} \right) + 
\kappa \, \frac{\rho^2}{n^2} \sum_{i=1}^N d_i^2 t_{ii}^2 
\sum_{j=1}^n \tilde d_j^2 \tilde t_{jj}^2
+ {\mathcal O}(n^{-1/2})\ , \\
&= 
\sum_{j=1}^n \left( 
\frac{\gamma(G_j - G_{j-1})}{{\bs\Delta}_j} + 
\frac{|\vartheta|^2 \underline\gamma(\underline{G}_j - \underline{G}_{j-1})}
{\underline{\bs\Delta}_j} \right)  \\
&\phantom{=} + 
\sum_{j=1}^n \left( 
\frac{2(F_j - F_{j-1})(1-F_j)}{{\bs\Delta}_j} 
+ 
\frac{2\re\left(\vartheta(\underline{F}_j - \underline{F}_{j-1})
(1-\bar\vartheta\bar{\underline{F}}_j) \right)}
{\underline{\bs\Delta}_j} \right)  \\
&\phantom{=} + 
\frac{\rho^2}{n} \sum_{j=1}^n \tilde d_j^2 \tilde t_{jj}^2 \left(
\frac{\gamma}{{\bs\Delta}_j} + 
|\vartheta|^2 \frac{\underline\gamma}{\underline{\bs\Delta}_j} 
+ \frac{\kappa}{n} \sum_{i=1}^N d_i^2 t_{ii}^2
\right) + {\mathcal O}(n^{-1/2}) \ ,\\ 
&\stackrel{\triangle}= Z_{1,n} + Z_{2,n} + Z_{3,n} + {\mathcal O}(n^{-1/2})  \ .
\end{align*} 
We prove in the sequel that $Z_{1,n} \geq 0$, $Z_{2,n} \geq 0$, and
that $\liminf_n Z_{3,n} > 0$.  It has already been noticed that
$\underline{\bs\Delta}_j \geq {\bs\Delta}_j$; moreover, it can be
proven by direct computation that $|\underline{G}_j -
\underline{G}_{j-1} | \leq {G}_j - {G}_{j-1}$, hence $Z_{1,n} \geq 0$.
As
\[
(1-F_j) - \frac{\gamma (M_j + G_j)}{1-F_j} 
\leq 
\left| 1 - \vartheta \underline{F}_j \right| 
- 
\frac{|\vartheta|^2 \underline\gamma \left(M_j + \underline{G}_j \right)}
{\left| 1 - \vartheta \underline{F}_j \right|} , 
\] 
this implies that $\underline{\bs\Delta}_j^{-1} |
1-\vartheta{\underline{F}}_j | \leq {\bs\Delta}_j^{-1}
(1-F_j)$. Noticing in addition that $|\underline{F}_j -
\underline{F}_{j-1} | \leq {F}_j - {F}_{j-1}$, we get $Z_{2,n} \geq
0$. The cumulant $\kappa= \E |X_{11}|^4 -2 - |\vartheta|^2$ satisfies
$\kappa \geq -1 - |\vartheta|^2$, hence
\begin{align*}
Z_{3,n} &\geq  
\frac{\rho^2}{n^2} \sum_{j=1}^n \tilde d_j^2 \tilde t_{jj}^2 
\left( \left( \frac{1}{{\bs\Delta}_j} -1 \right) 
+  |\vartheta|^2 \left( \frac{1}{\underline{\bs\Delta}_j} - 1 \right) 
\right) \sum_{i=1}^N d_i^2 t_{ii}^2 \\
&\phantom{=} 
+ 
\frac{\rho^2}{n} \sum_{j=1}^n \tilde d_j^2 \tilde t_{jj}^2 
\Bigl( 
\frac{1}{n {\bs\Delta}_j} 
\sum_{\scriptstyle{k,\ell=1} \atop \scriptstyle{k\neq\ell}}^N 
d_k^2 \left| t_{kl} \right|^2 
+ 
\frac{1}{n \underline{\bs\Delta}_j} 
\sum_{\scriptstyle{k,\ell=1} \atop \scriptstyle{k\neq\ell}}^N 
d_k^2 \left( t_{kl} \right)^2 \Bigr)\ , \\
&\geq  
\frac{\rho^2}{n^2} \sum_{j=1}^n \tilde d_j^2 \tilde t_{jj}^2 
\left( \left( \frac{1}{{\bs\Delta}_j} -1 \right) 
+  |\vartheta|^2 \left( \frac{1}{\underline{\bs\Delta}_j} - 1 \right) 
\right) \sum_{i=1}^N d_i^2 t_{ii}^2 
= 
\frac{\rho^2}{n^2} \sum_{j=1}^n \tilde d_j^2 \tilde t_{jj}^2 p_j 
\sum_{i=1}^N d_i^2 t_{ii}^2\ . 
\end{align*} 
As the term $p_j$ is linear in $|\vartheta|^2 \in [0,1]$, 
$p_j \geq \min \left( {\bs\Delta}_j^{-1}(1-{\bs\Delta}_j), 
{\bs\Delta}_j^{-1} + \underline{\bs\Delta}_j^{-1} - 2 \right)$.
We have 
\begin{eqnarray*}
{\bs\Delta}_j \underline{\bs\Delta}_j &\leq& 
\left( (1-F_j)^2 - \gamma(M_j+G_j) \right)
\left( (1+F_j)^2 + \gamma(M_j+G_j) \right) \ ,\\
&=& (1-F_j^2)^2 - \gamma^2 (M_j+G_j)^2 - 4 \gamma F_j (M_j+G_j) 
\quad \leq\quad  1\ .
\end{eqnarray*} 
Hence ${\bs\Delta}_j^{-1} + \underline{\bs\Delta}_j^{-1} - 2 
\geq {\bs\Delta}_j^{-1} + {\bs\Delta}_j - 2 = 
{\bs\Delta}_j^{-1} ( 1 - {\bs\Delta}_j)^2$. As $1 - {\bs\Delta}_j \geq 
\gamma M_j$, we get $p_j \geq \gamma^2 M_j^2$, which implies that
\begin{multline*}
Z_{3,n} \geq 
\frac{\rho^2\gamma^2}{n^2} \sum_{j=1}^n \tilde d_j^2 \tilde t_{jj}^2 M_j^2 
\sum_{i=1}^N d_i^2 t_{ii}^2 
= 
\gamma^2 \Bigl( \sum_{j=1}^n M_j^2 ( M_j - M_{j-1} ) \Bigr) 
\frac 1n \sum_{i=1}^N d_i^2 t_{ii}^2  \\ 
= 
\frac{\gamma^2}{3} \Bigl( \sum_{j=1}^n M_j^3 - M_{j-1}^3 \Bigr) 
\frac 1n \sum_{i=1}^N d_i^2 t_{ii}^2 + {\mathcal O}(n^{-1}) 
= 
\frac{\gamma^2 M_n^3}{3} 
\frac 1n \sum_{i=1}^N d_i^2 t_{ii}^2 
+ {\mathcal O}(n^{-1}) \ ,
\end{multline*}
whose liminf is positive by Lemma \ref{lm:estimates}. 

The estimates over the variance are therefore established. This completes the proof of Theorem \ref{th:CLT}.

\section{Proposition \ref{prop:bias} (bias): Main steps of the proof} 
\label{sec:proof-bias} 

Proof of Proposition \ref{prop:bias}-({\rm i}) can be found in 
\cite[Theorem 2]{Dumont10}. Let us prove ({\rm ii}). The same arguments as in 
the companion article \cite{HLN08} allow to write the bias term as:
$$
N \left( \E \mathcal{I}_n(\rho) - V_n(\rho)\right) = 
\int^\infty_\rho \tr \left( T(-\omega) - \E Q(-\omega) \right) 
{\bf d}\omega , 
$$
Recall that in the centered case where $A = 0$, $T(-\omega)$ and 
$\tilde T(-\omega)$ take the simple form 
$T(-\omega) = [ \omega(I_N +\tilde \delta(-\omega) D) ]^{-1}$ and 
$\tilde T(-\omega) = [ \omega(I_N + \delta(-\omega) \tilde D) ]^{-1}$, 
which implies that $\underline{\gamma} = \gamma$ and 
$\underline{\tilde \gamma} = \tilde \gamma$. 
We introduce the following intermediate quantities: 
\begin{align*} 
\alpha(-\omega) &= \frac 1n \tr D\E Q(-\omega) , & 
\tilde \alpha(-\omega)  &= \frac 1n \tr \tilde D\E \tilde Q(-\omega) , \\ 
C(-\omega) &= \Bigl( \omega(I_N + \tilde \alpha(-\omega)D) \Bigr)^{-1} , & 
\tilde C(-\omega) &= \Bigl( \omega(I_n + \alpha(-\omega)\tilde D) \Bigr)^{-1}.
\end{align*} 
From Theorem \ref{theo:quadra}, $n^{-1} \tr U(C-T) \to 0$ and 
$n^{-1} \tr \tilde U(\tilde C-\tilde T) \to 0$ for any sequences of
deterministic matrices $U$ and $\tilde U$ with bounded spectral norms. 

The proof consists of two steps: 
\subsection{Step 1} 
Given $\omega > 0$, let 
\[
\beta_n(\omega) = 
\left( \kappa 
+ \frac{|\vartheta|^2}{1 - \omega^2 |\vartheta|^2 \gamma \tilde\gamma} 
\right) R_n(\omega)  
\] 
where
\[
  R_n(\omega) = 
\frac{\frac{\omega^3}{n}\tr D T^2}{1 - \omega^2 \gamma \tilde\gamma}
\left( \frac{\gamma}{n} \tr (\tilde D\tilde T)^3 
- \frac{\omega\tilde\gamma^2}{n} \tr (DT)^3\right) 
- \frac{\omega^2 \tilde\gamma}{n} \tr D^2 T^3 . 
\] 
The purpose of this step is to show that $\int_\rho^\infty 
| \beta_n(\omega) | \, {\bf d}(\omega) < \infty$ and that 
\begin{equation} 
\label{step1-bias} 
N(\E {\mathcal I}_n(\rho) - V_n(\rho) ) - \int_\rho^\infty \beta_n(\omega) 
\, {\bf d}(\omega) 
\xrightarrow[n\to\infty]{} 0\ . 
\end{equation} 
By inspecting the expression of $\beta_n(\omega)$, by using 
Lemmas \ref{lm:estimates} and \ref{lm:var-bounds} and by recalling that
$1 - \omega^2 \gamma \tilde\gamma = \Delta_n$ taken at $z=-\omega$, we obtain
after a small derivation 
that $| \beta_n(\omega) | \leq K / \omega^3$ on $[\rho, \infty)$ where $K$ 
does not depend on $n$ nor on $\omega$. This proves the integrability
of $|\beta_n(\omega)|$. 
By taking up the poof of \cite[Inequality (7.10)]{HLN08} with minor 
modifications, we also show that 
\[
\left| \tr \left( T(-\omega) - \E Q(-\omega) \right) \right| 
\leq \frac{K}{\omega^2} 
\]
hence, showing 
\begin{equation}
\label{ident:bias}
\tr \left( T(-\omega) - \E Q(-\omega) \right) - \beta(\omega) 
\xrightarrow[n\rightarrow\infty]{} 0
\end{equation}
and applying the Dominated Convergence Theorem leads to \eqref{step1-bias}. 
In order to show \eqref{ident:bias}, we start by writing for $z=-\omega$
$$
\tr \left( T - \E Q \right) = \tr \left( T - C \right) 
                                + \tr \left( C - \E Q \right).
$$
Using the decomposition 
$\tr \left( T - C \right) = \tr C( C^{-1}- T^{-1})T = 
\omega \left(\tilde \alpha - \tilde \delta \right) \tr D C T$, we obtain 
\begin{equation}
\label{eq:bias1}
\tr \left( T - \E Q \right) = 
\omega n \left(\tilde \alpha - \tilde \delta \right) 
\frac 1n \tr DCT + \tr (C - \E Q)\ .
\end{equation}
On the other hand, writing $n(\alpha - \delta) = \tr D(\E Q - C) + 
\tr D(C-T) = \tr D(\E Q - C) - \omega (\tilde\alpha - \tilde\delta) 
\tr D^2 CT$ and similarly for $n(\tilde\alpha - \tilde\delta)$, we  
obtain the system 
\begin{equation}
\label{sys-alpha} 
\begin{bmatrix}
1 & \omega n^{-1} \tr D^2 CT \\
\omega n^{-1} \tr \tilde D^2 \tilde C \tilde T & 1 
\end{bmatrix} 
\begin{bmatrix} 
n(\alpha - \delta) \\ n(\tilde\alpha - \tilde\delta) 
\end{bmatrix} 
= 
\begin{bmatrix} 
\tr D(\E Q - C) \\ \tr \tilde D(\E\tilde Q - \tilde C)  
\end{bmatrix} . 
\end{equation} 
Consequently, in order to show \eqref{ident:bias}, we need to look for 
approximations of $\tr U (\E Q - C)$ and 
$\tr \tilde U (\E \tilde Q - \tilde C)$ for 
deterministic matrices $U$ and $\tilde U$ with bounded spectral norms: 

\begin{lemma}
\label{lm:tr(EQ-C)}
Assume that the setting of Proposition \ref{prop:bias} holds true. 
Fix $z=-\omega < 0$ and let $(U_n)_n$ (resp. $(\tilde U_n)_n$) be a 
sequence of $N \times N$ (resp. $n\times n$) diagonal deterministic matrices 
such that $\sup_n \max( \| U_n\|, \| \tilde U_n \|)  < \infty$. Then,
\begin{gather} 
\tr U_n(C - \E Q) 
+ \kappa \frac{\omega^2 \tilde \gamma}{n}  \tr U D^2 T^3   
+  |\vartheta|^2 \frac{\omega^2 \tilde\gamma}
            {1 - \omega^2 |\vartheta|^2 \gamma \tilde\gamma} 
\frac 1n \tr U D^2 T^3 
\xrightarrow[n\to\infty]{}  0, \label{E: trU(C - EQ)} \\
\tr \tilde U_n(\tilde C - \E \tilde Q) 
+ \kappa \frac{\omega^2 \gamma}{n}  \tr \tilde U \tilde D^2 \tilde T^3   
+  |\vartheta|^2 \frac{\omega^2 \gamma}
            {1 - \omega^2 |\vartheta|^2 \gamma \tilde\gamma} 
\frac 1n \tr \tilde U \tilde D^2 \tilde T^3 
\xrightarrow[n\to\infty]{}  0 \label{E_tilde: tr U(C - EQ)}. 
\end{gather} 
\end{lemma}
Recalling that $n^{-1} \tr U(C - T) \to 0$ and taking $U = D^2 T$, we have
$n^{-1} \tr D^2 CT - \gamma \to 0$. Similarly,  
$n^{-1} \tr \tilde D^2 \tilde C\tilde T - \tilde \gamma \to 0$. 
Solving system \eqref{sys-alpha} and using this lemma with $U = D$ and 
$\tilde U = \tilde D$, we obtain
\begin{align*} 
n(\tilde\alpha - \tilde\delta) &= 
\frac{1}{1-\omega^2 \gamma\tilde\gamma} 
 \left( 
 - \omega \tilde\gamma \tr D(\E Q - C) 
 + \tr \tilde D(\E\tilde Q - \tilde C) \right) + \varepsilon  \\
&= 
\frac{\kappa}{1 - \omega^2 \gamma \tilde \gamma}
\left( \frac{\omega^2 \gamma}{n} \tr (\tilde D\tilde T)^3 -
\frac{\omega^3 \tilde\gamma^2}{n} \tr (DT)^3 \right) \\
&\phantom{=}  + \frac{|\vartheta|^2}
{(1 - \omega^2 \gamma \tilde \gamma) 
(1 - \omega^2 |\vartheta|^2 \gamma \tilde \gamma)} 
\left( \frac{\omega^2 \gamma}{n} \tr (\tilde D \tilde T)^3 - \frac{\omega^3 \tilde\gamma^2}{n} \tr (DT)^3  
 \right) + \varepsilon' 
\end{align*}
where $\varepsilon, \varepsilon' \to 0$. Using Lemma \ref{lm:tr(EQ-C)} 
again in \eqref{eq:bias1} with $U = I$, we obtain \eqref{ident:bias}. 

The remainder of this paragraph is devoted to the proof of Lemma 
\ref{lm:tr(EQ-C)}. 

Recall the following notations:
$$
\tilde b_j(-\omega) = 
\frac 1{\omega\left(1+\frac{\tilde d_j}{n} \tr D Q_j(-\omega)\right)}
\quad \text{and} \quad 
e_j(-\omega) = 
\eta_j^* Q_j(-\omega) \eta_j - \frac{\tilde d_j}{n} \tr D Q_j(-\omega)  \ .
$$
Starting with 
\[
\tr U\left(\E Q - C \right) = 
\tr \E[ UC(C^{-1} - Q^{-1})Q ] 
= - \tr\E[ UC\Sigma\Sigma^* Q ] + \omega\tilde\alpha \tr \E [UCDQ] 
\] 
and using \eqref{eq:woodbury2} and \eqref{eq:diff-q-b}, we obtain 
$\tr U\left(\E Q - C \right) = Z_1 + Z_2 + Z_3$ where 
\begin{align*}
Z_1 &= \sum_{j=1}^n \E\left[ \omega^2 \tilde b_j^2 \, e_j \, 
\eta^*_j Q_j U C \eta_j \right] , \\
Z_2 &= - \sum_{j=1}^n \E\left[ \omega^3 \tilde q_{jj} \tilde b_j^2 \, e_j^2 \, 
\eta^*_j Q_j U C \eta_j \right] , \\
Z_3 &= 
\omega\, \tilde \alpha\, \E \tr U C D Q 
- \frac{\omega}{n} \sum^n_{j=1} \E\, \tilde b_j \tilde d_j  \tr D Q_j U C \ . 
\end{align*} 
In the remainder, we omit the study of the negligible terms to focus on the 
deterministic equivalent formulas; in this spirit, we shall denote by 
$\varepsilon$ a negligible term whose value might change from line to line. 

The term $Z_1$ is  
\[
Z_1 = \sum_{j=1}^n \E\left[ \omega^2 \tilde b_j^2 \, e_j \, 
\left( \eta^*_j Q_j U C \eta_j - \tilde d_j \frac{\tr D Q_j U C}{n}
\right) \right] . 
\]
Using Identity \eqref{eq:identite-fondamentale} with $M = Q_j$, $P= Q_j U C$ 
and ${\bs u} = 0$, we obtain:
\begin{multline*}
Z_1  = 
\frac 1n \sum^{n}_{j=1} 
\E \left[\omega^2 \tilde b_j^2  
\left(\frac{\tilde d_j^2}{n} \tr Q_j D Q_j U C D + 
\frac{|\vartheta|^2}{n} \tilde d_j^2 \tr Q_j D C U \bar Q_j D
\right.\right. \\
\left.\left.   
+ \frac{\kappa}{n} 
\sum^{N}_{i=1} \tilde d^2_j d_i^2 [Q_j]_{ii}[Q_j U C]_{ii}\right)\right] + \varepsilon .  
\end{multline*}
It is not difficult to check that 
\[
Z_2 = - \sum_{j=1}^n \E\left[ \omega^3 \tilde b_j^3 \, 
\tilde d_j \frac{\tr D Q_j U C}{n} \, e_j^2 \, 
\right] + \varepsilon . 
\]
Turning to $Z_3$, we have 
\begin{eqnarray*}
Z_3 &=& \frac{\omega}{n} \sum ^n_{j=1} 
\E\left[ 
\tilde d_j \E \tilde q_{jj} \tr U C D Q - \tilde d_j \tilde b_j \tr U C D Q_j 
\right] \\
&=& \frac{\omega}{n} \sum ^n_{j=1} \tilde d_j 
\E\left[ \E \tilde q_{jj} \left( \tr U C D Q_j - \omega \tilde q_{jj}\tr U C D Q_j \eta_j\eta^*_j Q_j \right) - \tilde b_j \tr U C D Q_j\right]\\
&=& 
\frac \omega n \sum^{n}_{j=1} \tilde d_j 
\E\left[(\E \tilde q_{jj} - \tilde b_j) \tr U C D Q_j\right] 
- \frac{\omega^2}{n} \sum^{n}_{j=1} \tilde d_j 
\E \tilde q_{jj} \E\left[\tilde b_j  \eta^*_j Q_j U C D Q_j \eta_j \right] 
+ \varepsilon . 
\end{eqnarray*}
Replacing $\tilde b_j$ by $\tilde q_{jj} + \omega \tilde b_j^2 e_j - \omega^2
\tilde q_{jj} \tilde b_j^2 e_j^2$, we have
\[
\E\left[(\E \tilde q_{jj} - \tilde b_j) \tr U C D Q_j\right] = 
\E\left[(\E \tilde q_{jj} - \tilde q_{jj}) \tr U C D Q_j\right] 
+ 
\E\left[ \omega^2 \tilde q_{jj} \tilde b^2_j e^2_j \, \tr U C D Q_j\right] . 
\] 
Since 
\[
\tilde q_{jj} - \E \tilde q_{jj} 
= 
\frac{\E(\eta_j^* Q_j \eta_j) - \eta_j^* Q_j \eta_j}
{\omega(1+\eta_j^* Q_j \eta_j)(1+\E(\eta_j^* Q_j \eta_j))} 
+ 
\E\left[ 
\frac{\eta_j^* Q_j \eta_j - \E(\eta_j^* Q_j \eta_j)}
{\omega(1+\eta_j^* Q_j \eta_j)(1+\E(\eta_j^* Q_j \eta_j))} 
\right]  
\]
we have $\E(\tilde q_{jj} - \E \tilde q_{jj})^2 = {\mathcal O}(1/n)$. Hence
\[
\E\left[(\E \tilde q_{jj} - \tilde q_{jj}) \tr U C D Q_j\right] 
= 
\E\left[(\E \tilde q_{jj} - \tilde q_{jj}) ( \tr U C D Q_j - 
\E \tr U C D Q_j ) \right] = {\mathcal O}(n^{-1/2}) 
\]
by Theorem \ref{theo:quadra}-\eqref{var(trUQ)}. It follows that
$\E\left[(\E \tilde q_{jj} - \tilde b_j) \tr U C D Q_j\right] = 
\E\left[ \omega^2 \tilde b^3_j e^2_j \, \tr U C D Q_j\right] + \varepsilon$,
hence 
\[
Z_2 + Z_3 = 
- \frac{1}{n} \sum^{n}_{j=1} 
\E\left[\omega^2 \tilde b^2_j \tilde d_j^2\frac 1n \tr Q_j D Q_j U C D 
\right] + \varepsilon  . 
\] 
Taking the sum $Z_1+Z_2+Z_3$, the terms that do not depend on $\vartheta$
nor on $\kappa$ cancel out, and we are left with 
\[
\tr U(\E Q - C) = 
|\vartheta|^2 \omega^2 \tilde\gamma  
\frac 1n \E \left[\tr Q D C U \bar Q D \right] 
+ \kappa 
\omega^2 \tilde\gamma 
\frac 1n \tr U D^2 T^3 
+ \varepsilon . 
\]
where we relied on the usual approximations for the diagonal entries of the 
resolvent (see Lemma \ref{lemma:chi1}) to obtain the term in $\kappa$. 
We now briefly characterize the asymptotic behavior of 
$n^{-1} \tr \E \, Q D C U \bar Q D$. Starting with $Q = T + \omega \tilde\delta
T D Q - T \Sigma\Sigma^* Q$, we have 
\[
\frac 1n \tr \E Q D C U \bar Q D = 
\frac 1n \tr \E TDCU \bar Q D 
+ \frac{\omega\tilde\delta}{n} \tr \E TDQDCU \bar Q D 
- \frac 1n \tr \E T \Sigma\Sigma^* QDCU \bar Q D , 
\]
and 
\begin{align*}
- \frac 1n \tr \E T \Sigma\Sigma^* QDCU \bar Q D 
&= - \frac 1n \sum_{j=1}^n 
  \E \omega \tilde q_{jj} \eta_j^* Q_j DCU \bar Q D T \eta_j \\
&= 
- \frac 1n \sum_{j=1}^n 
  \E \omega \tilde q_{jj} \eta_j^* Q_j DCU \bar Q_j D T \eta_j \\
& 
\ \ \ \ \ \ \ \ \ \ \ \ \ \ \ \ \ \ \ \ \ \ \ \ \ \ \ \ \ \ \ \ 
+ \frac 1n \sum_{j=1}^n 
  \E \omega^2 \tilde q_{jj}^2 \eta_j^* Q_j DCU \bar Q_j \bar \eta_j  
\, \eta_j^T D T \eta_j \\
&= 
- \frac{\omega\tilde\delta}{n} \tr \E TDQDCU \bar Q D 
+ |\vartheta|^2 \omega^2 \gamma \tilde\gamma 
\frac{\tr \E QDCU\bar Q D}{n} + \varepsilon . 
\end{align*} 
We therefore get 
$$
\frac 1n \tr \, \E \, Q D C U \bar Q D 
= \frac{\frac 1n \tr \, U D^2 T^3}
{1 - \omega^2 |\vartheta|^2 \gamma \tilde \gamma} + \varepsilon 
$$
and Convergence \eqref{E: trU(C - EQ)} of lemma \ref{lm:tr(EQ-C)} is shown. 
Convergence \eqref{E_tilde: tr U(C - EQ)} is proven similarly. 
Lemma \ref{lm:tr(EQ-C)} is proven, and Step 1 of the proof of Proposition 
\ref{prop:bias}-(ii) is established. 

\subsection{Step 2} 
The purpose of this step is to show that 
\[
R(\omega) = \frac 12 \frac{\bf d}{{\bf d}\omega} 
\left( \omega^2 \gamma(-\omega) \tilde\gamma(-\omega) \right) . 
\] 
Plugging into the expression of $\beta_n(\omega)$, it is straightforward to 
show that $\int_\rho^\infty \beta_n(\omega) \, {\bf d}(\omega)$ 
coincides with ${\mathcal B}_n$ given by \eqref{def:bias}. 

Recall that 
\[
R(\omega) = 
\frac{\omega^3 \gamma}{1 - \omega^2 \gamma \tilde\gamma}
\frac{\tr D T^2}{n} 
\frac{\tr (\tilde D\tilde T)^3}{n}  
- \frac{\omega^4\tilde\gamma^2}{1 - \omega^2 \gamma \tilde\gamma}
\frac{\tr D T^2}{n} \frac{\tr (DT)^3}{n}  
- \omega^2 \tilde\gamma \frac{\tr D^2 T^3}{n} 
= R_1 + R_2 + R_3 
\] 
Our method is similar to \cite[\S V.B]{HLN08Ieee}. 
We start by showing that the derivatives of $\tilde\gamma(-\omega)$ and
$\gamma(-\omega)$ that we denote respectively as $\tilde\gamma'$ and $\gamma'$
are 
\begin{equation}
\label{eq:tgamma'} 
\begin{split} 
\tilde\gamma' &= \frac{{\bf d}\tilde\gamma(-\omega)}{{\bf d}\omega} = 
- \frac{2}{\omega} \tilde\gamma + 
\frac{2 \omega}{1 - \omega^2 \gamma \tilde\gamma}
\frac{\tr (\tilde D \tilde T)^3}{n} 
\frac{\tr D T^2}{n}  \\
\gamma' &= \frac{{\bf d}\gamma(-\omega)}{{\bf d}\omega} = 
- \frac{2}{\omega} \gamma + 
\frac{2 \omega}{1 - \omega^2 \gamma \tilde\gamma}
\frac{\tr (D T)^3}{n} 
\frac{\tr \tilde D \tilde T^2}{n} . 
\end{split}
\end{equation}
We have
\begin{equation}
\label{eq:g'} 
\tilde\gamma' = 
\frac 1n \sum_{j=1}^n \tilde d_j^2 
\frac{{\bf d}}{{\bf d}\omega} 
\left( \frac{1}{\omega^2 (1 + \tilde d_j \delta(-\omega))^2} \right)
= 
- \frac{2}{\omega} \tilde\gamma - 2 \omega \delta' 
\frac{\tr (\tilde D \tilde T)^3}{n} 
\end{equation} 
where we put $\delta' = {\bf d}\delta(-\omega) / {\bf d}\omega$. This 
derivative can be expressed as  
\[
\delta' = 
\frac 1n \sum_{i=1}^N d_i \frac{{\bf d}}{{\bf d}\omega} 
\left( \frac{1}{\omega (1 + d_i \tilde\delta(-\omega))} \right)
= 
- \frac{\delta}{\omega} - \omega \gamma \tilde\delta' 
\]
where $\tilde\delta' = {\bf d}\tilde\delta(-\omega) / {\bf d}\omega$. 
Similarly, $\tilde\delta' = - \omega^{-1} \tilde\delta - \omega 
\tilde\gamma \delta'$. Combining the two equations, we obtain
\begin{equation}
\label{eq:delta'} 
\delta' = \frac{\gamma\tilde\delta - \omega^{-1} \delta}
{1 - \omega^2 \gamma \tilde\gamma} . 
\end{equation} 
Since $T = [ \omega(I + \tilde\delta D) ]^{-1}$ and 
$\tilde T = [ \omega(I + \delta \tilde D) ]^{-1}$, we have
\begin{equation}
\label{eq:T} 
T = \omega^{-1} I_N - \tilde\delta D T, 
\quad 
\tilde T = \omega^{-1} I_n - \delta \tilde D \tilde T . 
\end{equation} 
This leads to 
\begin{equation} 
\label{eq:DT2} 
\frac 1n \tr D T^2 = \frac 1n \tr DT( ( \omega^{-1} I - \tilde\delta D T)  
= \omega^{-1} \delta - \gamma \tilde\delta , \ \text{and} \ 
\frac 1n \tr \tilde D \tilde T^2 = 
\omega^{-1} \tilde\delta - \tilde\gamma \delta . 
\end{equation} 
Combining with \eqref{eq:delta'} and \eqref{eq:g'}, we obtain the first 
equation of \eqref{eq:tgamma'}, the second being obtained similarly. Using
the first equation, the term $R_1$ can be expressed as 
\[
R_1 = \frac 12 \left( \omega^2 \gamma \tilde\gamma' + 2 \omega \gamma
\tilde\gamma \right) . 
\]
Turning to $R_2$, we have 
\begin{align*}
R_2 &= 
- \frac{\omega^4\tilde\gamma^2}{1 - \omega^2 \gamma \tilde\gamma}
\frac{\tr (DT)^3}{n}  
\frac{\tr D T(\omega^{-1} I - \tilde\delta DT)}{n} \\
&= 
- \frac{\omega^3\delta\tilde\gamma^2}{1 - \omega^2 \gamma \tilde\gamma}
\frac{\tr (DT)^3}{n}  
+ \frac{\omega^4\tilde\delta\gamma\tilde\gamma^2}
{1 - \omega^2 \gamma \tilde\gamma}
\frac{\tr (DT)^3}{n}  \\ 
&\stackrel{(a)}{=} 
\frac{\omega^3 \tilde\gamma}{1 - \omega^2 \gamma \tilde\gamma} 
\frac{\tr (DT)^3}{n} \frac{\tr \tilde D \tilde T^2}{n} 
- \omega^2 \tilde\gamma \tilde\delta \frac{\tr (DT)^3}{n} \\ 
&\stackrel{(b)}{=} \frac 12
\left( \omega^2 \gamma' \tilde\gamma + 2 \omega\gamma\tilde\gamma \right)
- \omega^2 \tilde\gamma \tilde\delta \frac{\tr (DT)^3}{n} 
\end{align*}
where $(a)$ is due to $\tilde\gamma \delta = \omega^{-1} 
\tilde\delta - n^{-1} \tr \tilde D \tilde T^2$, see \eqref{eq:DT2}, and $(b)$
is due to \eqref{eq:tgamma'}. Considering $R_3$, we have by \eqref{eq:T},  
\[
- \omega^2 \tilde\gamma \tilde\delta \frac{\tr (DT)^3}{n} 
- \omega^2 \tilde\gamma \frac{\tr D^2 T^3}{n} 
= 
- \omega^2 \tilde\gamma \frac{\tr D^2 T^2( T + \tilde\delta D T)}{n} 
= 
- \omega \gamma \tilde\gamma . 
\]
We therefore have 
$R(\omega) = 
0.5 \left(
\omega^2 \gamma' \tilde\gamma + \omega^2 \gamma \tilde\gamma' + 
2 \omega \gamma \tilde\gamma \right) 
= 0.5 ( \omega^2 \gamma \tilde\gamma )'$. 
Proposition \ref{prop:bias} is proven. 

\begin{appendix}

\section{Proofs for Section \ref{sec:notations}}  

\subsection{Proofs of Eq. \eqref{eq:t-tilde-jj} and Eq. \eqref{eq:T-cal(T)_j}} 
\label{expr-T_j} 

Proof of Eq. \eqref{eq:t-tilde-jj} mainly relies on matrix identity 
\eqref{eq:mx-inversion} and on the following identity for the inverse 
of a partitioned matrix (see for instance \cite[Section 0.7.3]{HorJoh90}):
\begin{equation}
\label{eq:partition-mat}
\textrm{If}\quad A=\left[
\begin{array}{cc}
a_{11} & A_{12} \\
A_{21} & A_{22} 
\end{array}\right]\ ,\qquad \textrm{then}\quad 
\left(A^{-1}\right)_{11} = 
\left( a_{11} -A_{12} A_{22}^{-1}A_{21}\right)^{-1}\ .
\end{equation}
To lighten the
computations, let us introduce the following notations:
$$
{\mathcal I} = \left( I_n +\tilde \delta D\right)^{-1},\qquad
\tilde{\mathcal I} = \left( I_{N-1} +\delta \tilde D_1\right)^{-1}\ .
$$

In order to express a diagonal element of $\tilde T$, say $\tilde
t_{11}$ (without loss of generality), let us first write:
$$
\tilde T = \left[
\begin{array}{cc}
-z(1+\delta \tilde d_1) + a^*_1 {\mathcal I} a_1 & a^*_1 {\mathcal I} A_1\\
A_1^* {\mathcal I} a_1 & -z\tilde{\mathcal I}^{-1} + A_1^* {\mathcal I} A_1
\end{array}
\right]^{-1}\ .
$$
Hence, according to \eqref{eq:partition-mat}:
\begin{eqnarray*}
\frac 1{\tilde t_{11}} & = & -z(1+\delta \tilde d_1) + a^*_1 {\mathcal I} a_1 -
a^*_1 {\mathcal I} A_1 \left[ -z\tilde{\mathcal I}^{-1} + A_1^* {\mathcal I} A_1 \right]^{-1} A_1^* {\mathcal I} a_1 \\
&\stackrel{(a)}=& -z(1+\delta \tilde d_1) + a^*_1 {\mathcal I} a_1 -
a^*_1 {\mathcal I} A_1 \left[ -\frac 1z \tilde {\mathcal I} 
+ \frac 1z \tilde {\mathcal I} A_1^* {\mathcal T}_1 A_1 \tilde {\mathcal I} 
\right] A_1^* {\mathcal I} a_1 \\
&\stackrel{(b)}= & -z(1+\delta \tilde d_1) + a^*_1 {\mathcal I} a_1 
+\frac 1z a_1^* {\mathcal I} \left( {\mathcal T}_1^{-1} +z{\mathcal I}^{-1}\right) {\mathcal I} a_1\\
&&\qquad \qquad -\frac 1z a_1^* {\mathcal I} A_1 \tilde {\mathcal I} A_1^* \left[
I_N +z {\mathcal T}_1 {\mathcal I}^{-1}
\right] {\mathcal I} a_1\\
&\stackrel{(b)}= & -z(1+\delta \tilde d_1) + 2 a^*_1 {\mathcal I} a_1 + \frac 1z a_1^* 
{\mathcal I} {\mathcal T}_1^{-1} {\mathcal I} a_1 \\
&&\qquad \qquad 
-\frac 1z a_1^* {\mathcal I} \left[ {\mathcal T}_1^{-1} +z{\mathcal I}^{-1} \right]  \left[
I_N +z {\mathcal T}_1 {\mathcal I}^{-1}
\right] {\mathcal I} a_1 \\
&=& -z(1+\delta \tilde d_1) - z a_1^* {\mathcal T}_1 a_1\ ,
\end{eqnarray*}
where $(a)$ follows from \eqref{eq:mx-inversion}, $(b)$ from equalities 
$$
{\mathcal T}_1 A_1 \tilde {\mathcal I} A_1^*  = I_N +z {\mathcal T}_1 {\mathcal I}^{-1}\quad \textrm{and}\quad 
A_1 \tilde {\mathcal I} A_1^*  = {\mathcal T}_1^{-1} +z{\mathcal I}^{-1}\ 
$$ 
which follow from the mere definition of ${\mathcal T}_1$. Finally,
\eqref{eq:t-tilde-jj} is established.

Let us now turn to the proof of \eqref{eq:T-cal(T)_j}. Notice first
that $T$ can be expressed as 
$$
T=\left[ -z(I_N +\tilde \delta D) +A_1 (I_{n-1} +\delta \tilde
  D_1)^{-1} A_1^*+ \frac{a_1 a_1^*}{1+\delta \tilde d_1} \right]^{-1}\ .
$$
Applying \eqref{eq:mx-inversion} readily yields:
$$
T= {\mathcal T}_1 - {\mathcal T}_1 a_1 \left( 1+\delta \tilde d_1
  +a_1^* {\mathcal T}_1 a_1\right) a_1^* {\mathcal T}_1\ .
$$
It remains to multiply by $a_1^*$ (left), $b$ (right) and to 
use \eqref{eq:t-tilde-jj} to establish \eqref{eq:T-cal(T)_j}. 

\subsection{Proof of Inequality \eqref{eq:BS-lemma-2}}
\label{bound-e} 
We provide here some elements to establish that $ \mathbb{E}|e_j|^p =
{\mathcal O}(n^{-p/2})$. Recall the definition \eqref{eq:def-e} of
$e_j$ and write:
$$
\mathbb{E}|e_j|^p \le K \left\{
\mathbb{E}\left| y_j^* Q_j y_j - \frac {\tilde d_j}n \mathrm{Tr} DQ_j\right|^p
+\mathbb{E}\left| a_j^* Q_j y_j \right|^p  + \mathbb{E}\left| y_j^* Q_j a_j \right|^p
\right\}\ .
$$
The first term of the r.h.s. can be directly estimated with the help of Lemma \ref{lm-approx-quadra}.
The two remaining terms are similar and can be estimated in the following way:
\begin{eqnarray*}
\mathbb{E}\left| a_j^* Q_j y_j \right|^p &=& \mathbb{E}\left( y_j^* Q^*_j a_j a_j^* Q_j y_j \right)^{p/2}\\
&\le & K \left( \mathbb{E}\left| y_j^* Q^*_j a_j a_j^* Q_j y_j - \frac{d_j^2}n \mathrm{Tr} Q_j^* a_j a_j^* Q_j \right|^{p/2} + \mathbb{E}\left| \frac 1n \mathrm{Tr} Q_j^* a_j a_j^* Q_j \right|^{p/2}
\right) \ .
\end{eqnarray*}
The first term of the r.h.s. can be handled with the help of Lemma
\ref{lm-approx-quadra} (notice that $Q^*_j a_j a_j^* Q_j$ is of rank one and
has a bounded spectral norm), and the second term is directly of the
right order.

\subsection{Proof of Theorem \ref{theo:quadra}}
\label{anx:th-quadra}

Items \eqref{quadra-sum-bounded}--\eqref{conv-quadra} of Theorem 
\ref{theo:quadra} are shown in \cite{HLNV10pre}. Let us show 
Theorem \ref{theo:quadra}-\eqref{eq-u(Qj-Tj)v}. 
Denote by $(\delta_j, \tilde\delta_j)$ the solution 
of System \eqref{eq:fundamental} when $A$ and $\tilde D$ are replaced with
$A_j$ and $\tilde D_j$ respectively. Let $T_j$ and $\tilde T_j$ be the 
matrices associated to $(\delta_j, \tilde\delta_j)$ as in Eq. \eqref{eq:def-T}. 
Then $\E| u^* (Q_j - T_j) v |^{2p} \leq K_p n^{-p}$ by Item 
\eqref{conv-quadra}, and we only need to show that $| u^* ( {\mathcal T}_j 
- T_j ) v | \leq K/\sqrt{n}$. 
We have 
\begin{multline*}
| \delta - \delta_j | = \frac 1n \left| \tr D ( T - T_j) \right| \\ 
\leq \frac 1n \left| \tr \E D ( T-Q ) \right| + 
\frac 1n \left| \tr \E D ( Q-Q_j ) \right| + 
\frac 1n \left| \tr \E D ( Q_j - T_j ) \right| 
= {\mathcal O}(n^{-1})
\end{multline*} 
by Item \eqref{speed(T-Q)} and Lemma \ref{lem:rank1-perturbation}. 
Moreover, 
\[
| \tilde\delta - \tilde\delta_j | 
\leq \frac 1n \left| \tr \E \tilde D ( \tilde T-\tilde Q ) \right| + 
\frac 1n 
    \left| \E ( \tr \tilde D \tilde Q- \tr \tilde D_j \tilde Q_j ) \right| + 
\frac 1n \left| \tr \E \tilde D_j ( \tilde Q_j - \tilde T_j ) \right| . 
\]
In order to deal with the middle term at the r.h.s., assume without generality
loss that $j=1$. Using the identity in \cite[Section 0.7.3]{HorJoh90} 
for the inverse of a partitioned matrix, we obtain 
\[
\tilde Q = \begin{bmatrix} 
\tilde q_{11} & \times \\
\times & \tilde Q_1 + \tilde q_{11} \tilde Q_1 \Sigma_1^* \eta_1 
\eta_1^* \Sigma_1 \tilde Q_1 
\end{bmatrix}
\] 
hence $\E ( \tr \tilde D \tilde Q- \tr \tilde D_1 \tilde Q_1 ) = 
{\mathcal O}(1)$, which shows that $| \tilde\delta - \tilde\delta_j |  = 
{\mathcal O}(n^{-1})$. We now have 
\begin{multline*} 
u^* ( {\mathcal T}_j - T_j ) v = 
u^* {\mathcal T}_j \left( T_j^{-1} - {\mathcal T}_j^{-1} \right) T_j v \\ 
= 
u^* {\mathcal T}_j \left( \rho (\tilde\delta_j - \tilde\delta) D + 
 (\delta - \delta_j) 
A_j (I+\delta_j \tilde D_j)^{-1} \tilde D_j (I+\delta \tilde D_j)^{-1} 
A_j^* \right)  T_j v 
= {\mathcal O}(n^{-1})   
\end{multline*} 
which proves Item \eqref{eq-u(Qj-Tj)v}. In order to prove Item 
\eqref{var(trUQ)}, we develop $\tr U(Q-\E Q)$ as a sum of martingale 
differences: 
\begin{multline*} 
\tr U(Q-\E Q) = \sum_{j=1}^n (\E_j - \E_{j-1}) \tr U Q \\ 
= \sum_{j=1}^n (\E_j - \E_{j-1}) \tr U ( Q - Q_j) 
= - \sum_{j=1}^n (\E_j - \E_{j-1})
( \rho \tilde q_{jj} \eta_j^* Q_j U Q_j \eta_j )
\end{multline*} 
by \eqref{eq:woodbury}, hence
$\E\left| \tr U(Q - \E Q) \right|^2 = 
\sum_{j=1}^n \E \left| (\E_j - \E_{j-1})
( \rho \tilde q_{jj} \eta_j^* Q_j U Q_j \eta_j ) \right|^2
$. We now use \eqref{eq:diff-q-b}. We have 
\begin{multline*} 
\E \left| (\E_j - \E_{j-1})
( \rho \tilde b_{j} \eta_j^* Q_j U Q_j \eta_j ) \right|^2 \\ 
= 
\E\left| \E_j \left( \rho \tilde b_{j} \left( 
\eta_j^* Q_j U Q_j \eta_j - \tilde d_j \frac{\tr D  Q_j U Q_j}{n} - 
a_j^*  Q_j U Q_j a_j \right) \right) \right|^2 = {\mathcal O}(n^{-1})  
\end{multline*} 
by Lemma \ref{lm-approx-quadra}, and furthermore, 
\[
\E \left| (\E_j - \E_{j-1})
( \rho^2 \tilde q_{jj} \tilde b_j e_j \eta_j^* Q_j U Q_j \eta_j ) \right|^2
\leq K \left( \E e_j^4 \ \E | \eta_j^* Q_j U Q_j \eta_j |^4 \right)^{1/2} 
= {\mathcal O}(n^{-1})  
\]
by  Lemma \ref{lm-approx-quadra} and \eqref{eq:BS-lemma-2}. This shows 
Th.\ref{theo:quadra}-\eqref{var(trUQ)}.

\subsection{Proof of Lemma \ref{lm:estimates}}  
\label{app:estimates-proof}
The two first upper bounds are easy to obtain, given
that $\delta_n$ and $\tilde \delta_n$ are Stieltjes transforms of
nonnegative measures with respective total mass $n^{-1} \tr D$ and
$n^{-1} \tr \tilde D$. Now $\tr DT^2 \le \dmax \, \tr T^2$ by Inequality 
\eqref{eq:von-neumann-lite}, which in turn is smaller than $N\dmax
\rho^{-2}$, hence the third upper bound, and the other upper bounds can be
proven similarly. Let us now prove the first lower bound.
\begin{eqnarray*}
\tr D &=& \tr (T^{\frac 12} D T^{\frac 12} T^{-1}) 
\ \stackrel{(a)}{\le}\  \tr (DT)\times \| T^{-1}\|\ ,\\
&\le& \tr (DT)\times \left( \rho(1 + \tilde \delta_n \dmax) + \amax^2 \| (I + \delta_n \tilde D)^{-1}\| \right)\ ,  \\
&\stackrel{(b)}\le& \tr (DT)\times \left( \rho + \dmax \dtmax + \amax^2 \right)\ ,
\end{eqnarray*}
where $(a)$ follows from \eqref{eq:von-neumann-lite} and $(b)$ from
the upper bound on $\tilde\delta_n$.  This readily yields $\delta_n$'s lower
bound and $\tilde \delta_n$'s lower bound which can be proven
similarly. Writing $\tr D \leq (\tr D T^2) \|T^{-1}\|^2$, we obtain the 
lower bounds on $n^{-1} \tr DT^2$ and $n^{-1} \tr \tilde D\tilde T^2$ 
similarly. The lower bound for $\gamma_n$ follows from the same ideas:
\begin{eqnarray*}
\left( \frac 1N \tr D\right)^2 &\le & \frac 1N \tr D^2\ = \ \frac 1N \tr (T^{\frac 12} D^2 T^{\frac 12} T^{-1})\ ,\\
&\le & \frac 1N \tr (T^{\frac 12} D^2 T^{\frac 12})\times  \| T^{-1}\| \ = \ 
\frac 1N \tr (T^{\frac 12}D T^{\frac 12}T^{-1}T^{\frac 12}D  T^{\frac 12})\times  \| T^{-1}\|\ , \\
&\le & \frac 1N \tr (TDTD)\times  \| T^{-1}\|^2\ ,
\end{eqnarray*}
and one readily obtains $\gamma_n$'s lower bound (and similarly
$\tilde \gamma_n$'s lower bound) using Assumption {\bf A-\ref{ass:D}}
and the upper estimate previously obtained for $\|T^{-1}\|$.

The two last series of inequalities related to $n^{-1} \sum_{i=1}^N
d_i^2 t_{ii}^2$ and $n^{-1} \sum_{j=1}^n \tilde d_j^2 \tilde t_{jj}^2$
can be proven with similar arguments (lower bounds are in fact easier
to obtain as one can directly get lower bounds for $t_{ii}$ and $\tilde t_{jj}$
- using \eqref{eq:t-tilde-jj} for instance).

\subsection{Proof of Lemma \ref{lm:var-bounds}} 
\label{app:bounds-proof}
From \eqref{eq:def-T}, 
$T A( I + \delta \tilde D)^{-1} A^* = I - \rho T (I + \tilde\delta D)$. 
Moreover, 
$( I + \delta \tilde D)^{-1}\tilde D 
= \delta^{-1} I - \delta^{-1} ( I + \delta \tilde D )^{-1}$. Hence
\begin{eqnarray}
\frac 1n \tr D^{1/2} T A(I + \delta \tilde D)^{-2} 
\tilde D A^* T D^{\frac 12} &\leq&
\frac{1}{n \delta} \tr D T A(I + \delta \tilde D)^{-1} A^* T \nonumber \\
&=& 1 - \frac{\rho}{n\delta} \tr D T^2 - \rho \frac{\tilde\delta}{\delta} \gamma \label{bnd-Fn} 
\end{eqnarray} 
which proves the first assertion with the help of the results of Lemma
\ref{lm:estimates}. Similarly, 
\begin{equation}
\label{bnd-tildeFn} 
\frac 1n \tr \tilde D^{1/2} \tilde T A^* (I + \tilde \delta D)^{-2} 
D A \tilde T \tilde D^{\frac 12} \leq
1 - \frac{\rho}{n\tilde \delta} \tr \tilde D \tilde T^2 
- \rho \frac{\delta}{\tilde\delta} \tilde\gamma
\end{equation} 
We now show that the left hand sides (l.h.s.) of \eqref{bnd-Fn} and 
\eqref{bnd-tildeFn} are equal. Using the well known matrix identity 
$(I+UV)^{-1} U = U(I+VU)^{-1}$, 
\begin{align*}
T A (I + \delta \tilde D)^{-1} &= 
\rho^{-1} ( I + \tilde\delta D)^{-1} 
\left( I + \rho^{-1} A (I + \delta \tilde D)^{-1} A^* (I + \tilde\delta D)^{-1}
\right)^{-1} 
A (I + \delta \tilde D)^{-1}  \\
&= 
\rho^{-1} ( I + \tilde\delta D)^{-1} 
 A (I + \delta \tilde D)^{-1} 
\left( I + \rho^{-1} A^* (I + \tilde\delta D)^{-1}A (I + \delta \tilde D)^{-1} 
\right)^{-1} \\ 
&= 
( I + \tilde\delta D)^{-1} A \tilde T . 
\end{align*} 
Plugging this identity in the l.h.s.~of \eqref{bnd-Fn}, and identifying with 
the l.h.s.~of \eqref{bnd-tildeFn}, we obtain the result. 
As a consequence, we have
\begin{align*} 
\left( 1 - \frac 1n \tr D^{1/2} T A(I + \delta \tilde D)^{-2} 
\tilde D A^* T D^{\frac 12} \right)^2 &\geq 
\left( 
\frac{\rho}{n\delta} \tr D T^2 + \rho \frac{\tilde\delta}{\delta} \gamma 
\right)
\left( \frac{\rho}{n\tilde \delta} \tr \tilde D \tilde T^2 
+ \rho \frac{\delta}{\tilde\delta} \tilde\gamma \right) \\ 
&\geq 
\rho^2 \gamma\tilde\gamma + 
\frac{\rho}{n\delta} \tr D T^2
\frac{\rho}{n\tilde \delta} \tr \tilde D \tilde T^2 
\end{align*} 
which is the second assertion. By Lemma \ref{lm:estimates}, this leads to 
$\liminf \Delta_n > 0$. Lemma \ref{lm:var-bounds} is proven.

\section{Additional proofs for Section \ref{sec:proof-clt-2}} 

\subsection{Proof of Lemma \ref{lem:control_variances}}\label{app:control_variances}
Let us show that $\max_j \var( a^* \E_j Q D \E_j Q a ) = {\mathcal O}(n^{-1})$. 
We have 
\begin{align}
a^* \E_j Q D \E_j Q a - \E a^* \E_j Q D \E_j Q a
&= \sum_{i=1}^j (\E_i - \E_{i-1}) 
(a^* \E_j Q D \E_j Q a ) \nonumber \\  
&= \sum_{i=1}^j (\E_i - \E_{i-1}) 
\left\| D^{1/2} \E_j( Q_i - \rho \tilde{q}_{ii} Q_i \eta_i \eta_i^* Q_i) 
a \right\|^2 \nonumber \\
&= \sum_{i=1}^j (\E_i - \E_{i-1}) 
\left[ -2 \rho \re\left( 
a^* (\E_j Q_i) D (\E_j \tilde{q}_{ii} Q_i \eta_i \eta_i^* Q_i ) a
\right)\phantom{D^{1/2}} \right. \nonumber \\ 
& \ \ \ \ \ \ \ \ \ \ \ \ \ \ \ \ \ \ \ \ \ \ \ \ \ \ \ \ \ \ \ \ \ \ \ 
\left.  
+ \| \E_j ( \rho\tilde{q}_{ii} \, \eta_i^* Q_i  a \, D^{1/2} Q_i \eta_i ) \|^2 
\right] \nonumber \\
&\stackrel{\triangle}= 2 \re(X) + Z\ , \label{eq-a*(EjQ)^2a}   
\end{align} 
and the variance of $a^* \E_j Q D \E_j Q a$ is the sum of the variances of 
these martingale increments. Consider the term $X$. 
Recalling that 
$\tilde{q}_{ii} = \tilde{b}_i - \rho \tilde{q}_{ii} \tilde{b}_i e_i$, 
\begin{align*} 
X &= -\rho \sum_{i=1}^j (\E_i - \E_{i-1}) 
( \tilde{b}_{i} \, \eta_i^* Q_i a a^*  (\E_j Q_i) D Q_i \eta_i )  
 + \rho^2 \sum_{i=1}^j (\E_i - \E_{i-1}) 
( \tilde{b}_{i} \tilde{q}_{ii} e_i 
\eta_i^* Q_i a a^* (\E_j Q_i) D Q_i \eta_i ) \\ 
&\stackrel{\triangle}= X_1 + X_2 \ . 
\end{align*} 
Let $M_i = Q_i a a^*  (\E_j Q_i) D Q_i$. The term $X_1$ satisfies 
\begin{equation} 
\E|X_1|^2 = 
\rho^2 \sum_{i=1}^j 
\E \left| \E_i \tilde{b}_{i} \left( 
y_i^* M_i y_i - \frac{\tilde d_i}{n} \tr D M_i + 
y_i^* M_i a_i + a_i^* M_i y_i 
\right) \right|^2  \ .
\label{eq:X1-variance} 
\end{equation} 
Since $M_i$ is a rank one matrix, 
$\sum_{i=1}^j \E | y_i^* M_i y_i - \tilde d_i \tr D M_i / n |^2 \leq K/n$. 
Moreover, 
\begin{multline*} 
\sum_{i=1}^j \E\left| y_i^* M_i a_i \right|^2 
= 
\frac 1n \sum_{i=1}^j \tilde d_i \E \left( a^* Q_i D Q_i a \ 
\left| a^*  (\E_j Q_i) D Q_i a_i \right|^2 \right) 
\leq \frac Kn \sum_{i=1}^j \E \left| a^*  (\E_j Q_i) D Q_i a_i \right|^2 . 
\end{multline*}
The summand at the r.h.s. of the inequality satisfies: 
\begin{align*} 
\left| a^*  (\E_j Q_i) D Q_i a_i \right|^2 
&\leq 4 \left( 
\left| a^*  (\E_j Q ) D Q a_i \right|^2 + 
\left| a^*  (\E_j(Q_i - Q)) D ( Q_i - Q) a_i \right|^2 + \right. \\ 
& \ \ \ \ 
 \ \ \ \ \ \ \ \ \ \ \ \ \ \ \ \ 
\left. \left| a^*  (\E_j Q ) D ( Q_i - Q) a_i \right|^2 + 
\left| a^*  (\E_j(Q_i - Q)) D Q a_i \right|^2 \right) \\
&\stackrel{\triangle}= 4( W_{i,1} + W_{i,2} + W_{i,3} + W_{i,4} ) . 
\end{align*}  
Recalling that $A_{1:j} = [a_1, \cdots, a_j ]$, we have 
$$
\sum_{i=1}^j W_{i,1}  = a^* (\E_j Q) D Q A_{1:j} 
A_{1:j}^* Q D (\E_j Q) a \leq K\ .
$$ 
Recalling \eqref{eq:Qj=f(Q,eta)} and \eqref{eq:1+etaQeta}, and writing $\xi_i = 1+\eta_i^* Q_i \eta_i$, 
we have: 
\[
W_{i,2} \leq  
\frac{a_i^* Q \eta_i \eta_i^* Q a_i}{1-\eta_i^* Q \eta_i} \times 
a^* \E_j\left( \xi_i \, Q \eta_i \eta_i^* Q \right) \ 
D \frac{Q \eta_i \eta_i^* Q}{1-\eta_i^* Q \eta_i} D \ 
\E_j\left( \xi_i \, Q \eta_i \eta_i^* Q \right) a \ .
\]
As $\| (1- \eta_i^* Q \eta_i)^{-1} Q \eta_i \eta_i^* Q \| = 
\| Q - Q_i \| \leq K$ and $\| Q \Sigma \| \leq K$, we have 
$\sum_{i=1}^j \E W_{i,2} \leq 
\sum_{i=1}^j \E \left[ | a^* Q \eta_i |^2 | \xi_i |^2 \right]$. Writing
$\xi_i = (\rho\tilde b_i)^{-1} + e_i$, and noticing that 
$(\rho\tilde b_i)^{-1}$ is bounded, we obtain:
\[
\sum_{i=1}^j \E W_{i,2} \leq 
2 \E a^* Q \Sigma_{1:j} \diag ( (\rho \tilde b_1)^{-1},\dots,
(\rho \tilde b_j)^{-1}) \Sigma_{1:j}^* Q a \ + \ 
K \sum_{i=1}^j \E| e_i |^2 
\leq K . 
\]
The terms $W_{i,3}$ and $W_{i,4}$ can be handled by similar derivations. 
\\
We get that  
$\sum_{i=1}^j \E | y_i^* M_i a_i |^2 \leq {K} / {n}$. 
The terms $a_i^* M_i y_i$ on the right hand side of \eqref{eq:X1-variance} 
satisfy 
$\sum_{i=1}^j \E \left| a_i^* M_i y_i \right|^2 
\leq 
K n^{-1} \sum_{i=1}^j \E | a_i^* Q_i a |^2 \leq K / n$, 
which proves that $\E|X_1|^2 \leq K/n$. 

We now consider $X_2$, which satisfies $\E|X_2|^2 \leq 2\rho^4
\sum_{i=1}^j \E | \tilde{b}_{i} \tilde{q}_{ii} e_i \eta_i^* M_i \eta_i
|^2$. We have
\begin{eqnarray*} 
\sum_{i=1}^j 
\E \left| \tilde{b}_{i} \tilde{q}_{ii} e_i a_i^* M_i a_i \right|^2 
&\leq &K \sum_{i=1}^j 
\E \E^{(i)} \left| e_i a_i^* M_i a_i \right|^2  \\ 
&\leq& \frac Kn \sum_{i=1}^j \E| a_i^* M_i a_i |^2 
\leq \frac Kn \sum_{i=1}^j \E| a_i^* Q_i a |^2 
\leq \frac Kn \ , 
\end{eqnarray*} 
where 
$\E^{(i)} = \E[ \cdot | y_1, \ldots, y_{i-1}, y_{i+1}, \ldots, y_n]$. 
Moreover, 
\[
\sum_{i=1}^j 
\E \left| \tilde{b}_{i} \tilde{q}_{ii} e_i y_i^* M_i a_i \right|^2 
\leq 
K \sum_{i=1}^j (\E | e_i |^4)^{1/2} (\E | y_i^* M_i a_i |^4)^{1/2}  
\leq \frac Kn 
\]
and similarly for the terms in $a_i^* M_i y_i$ and in $y_i^* M_i y_i$. 
We get that $\E|X_2|^2 = {\mathcal O}(n^{-1})$. 
We now turn to the term $Z$ of equation \eqref{eq-a*(EjQ)^2a}. 
To control the variance of $Z$, we only need to control the
variances of the terms: 
\begin{align*} 
Z_1 &= \sum_{i=1}^j (\E_i - \E_{i-1})
(\E_j \tilde{q}_{ii} a^* Q_i y_i \eta_i^* Q_i ) D 
(\E_j \tilde{q}_{ii} Q_i \eta_i a_i^* Q_i a )  , \\
Z_2 &= \sum_{i=1}^j (\E_i - \E_{i-1})
\| \E_j(\tilde{q}_{ii}  y_i^* Q_i a \, D^{1/2} Q_i \eta_i)\|^2 , \\
Z_3 &= \sum_{i=1}^j (\E_i - \E_{i-1})
\| \E_j(\tilde{q}_{ii}  a_i^* Q_i a \, D^{1/2} Q_i \eta_i)\|^2 . 
\end{align*}
The first term satisfies 
\begin{eqnarray*}
\E|Z_1|^2 &\leq& 2 \sum_{i=1}^j \E  
\left| (\E_j \tilde{q}_{ii} a^* Q_i y_i \eta_i^* Q_i ) D 
(\E_j \tilde{q}_{ii} Q_i \eta_i a_i^* Q_i a ) \right|^2 \\
&\leq& 
2 \sum_{i=1}^j \E 
\left| (\E_j a^* Q_i y_i \eta_i^* Q_i ) D Q_i \eta_i a_i^* Q_i a  \right|^2 \\
&=& 2 \sum_{i=1}^j \E \left[ | a_i^* Q_i a |^2 \, \E^{(i)} 
\left| (\E_j a^* Q_i y_i \eta_i^* Q_i ) D Q_i \eta_i \right|^2 \right] \\
&\leq& \frac{K}{n} \sum_{i=1}^j \E | a_i^* Q_i a |^2 
\quad = \quad {\mathcal O}(n^{-1}) \ ,
\end{eqnarray*}
where the second inequality comes from $\E | \E_j(X) \E_j(Y) |^2 = 
\E | \E_j (X \E_j(Y)) |^2 \leq \E | X \E_j(Y) |^2$. 
The terms $Z_2$ and $Z_3$ can be handled similarly; details are omitted.
The result is $\E|Z|^2 \leq K/n$. 

Hence, $\var\left( a^* (\E_j Q)^2 a \right) = {\mathcal O}(n^{-1})$.
The estimate $\var\left( \tr (\E_j Q) D (\E_j Q) \right) = {\mathcal
  O}(1)$ can be established similarly.

\subsection{Proof of Lemma \ref{lm-E(aQa)}} 
\label{anx-lemma-E(aQa)} 

Recalling the expression \eqref{eq:t-tilde-jj} of $\tilde t_{\ell\ell}$, 
we notice that 
$(1 - \rho\tilde t_{\ell\ell} a_\ell^* {\mathcal T}_{\ell} a_\ell)
= \rho\tilde t_{\ell\ell} (1 + \tilde d_\ell \delta)$ is bounded below. 
It follows from Theorem \ref{theo:quadra}-\eqref{conv-quadra} that 
\[
\sum_{\ell=1}^j 
\frac{\alpha_\ell \, u^* T a_\ell \, \E\left[ a_\ell^* Q u \right]}
{1 - \rho\tilde t_{\ell\ell} a_\ell^* {\mathcal T}_{\ell} a_\ell} 
= 
\sum_{\ell=1}^j 
\frac{\alpha_\ell \, u^* T a_\ell \, a_\ell^* T u }
{\rho\tilde t_{\ell\ell} (1 + \tilde d_\ell \delta)} 
+ {\mathcal O}(n^{-1/2}) 
\]
Moreover, 
\begin{multline*} 
\sum_{\ell=1}^j \alpha_\ell \, u^* T a_\ell 
\left(\frac{\E\left[ a_\ell^* Q u \right]}
{1 - \rho\tilde t_{\ell\ell} a_\ell^* {\mathcal T}_{\ell} a_\ell} 
- \E\left[ a_\ell^* Q_\ell u \right] \right) 
= 
\sum_{\ell=1}^j \alpha_\ell \, u^* T a_\ell 
\E \left[ a_\ell^* Q_\ell u 
\left( \frac{1 - \rho\tilde q_{\ell\ell} a_\ell^* Q_\ell \eta_\ell}
{1 - \rho\tilde t_{\ell\ell} a_\ell^* {\mathcal T}_{\ell} a_\ell} - 1 
\right) \right] \\
- \sum_{\ell=1}^j \alpha_\ell \, u^* T a_\ell 
\frac{\E[ \rho\tilde q_{\ell\ell} a_\ell^* Q_\ell \eta_\ell \, 
y_\ell^* Q_\ell u ]}
{1 - \rho\tilde t_{\ell\ell} a_\ell^* {\mathcal T}_{\ell} a_\ell}
= \varepsilon_1 + \varepsilon_2 
\end{multline*} 
We have $\varepsilon_1 = \sum_{\ell=1}^j \alpha_\ell \, u^* T a_\ell 
\E \left[ a_\ell^* Q_\ell u \, \xi_\ell \right]$ where 
$\E | \xi_\ell |^p \leq K n^{-p/2}$ for $p\geq 2$. It follows that 
$| \varepsilon_1 | \leq \left( \sum_{\ell=1}^j \alpha_\ell^2 
| u^* T a_\ell |^2 \E \xi_\ell^2 \right)^{1/2} 
\left( \sum_{\ell=1}^j \E | a_\ell^* Q_\ell u |^2 \right)^{1/2} 
\leq K / \sqrt{n}$ by Theorem \ref{theo:quadra}-\eqref{quadra-sum-bounded}. 
By writing 
\[
\varepsilon_2 = 
- \sum_{\ell=1}^j \alpha_\ell \, u^* T a_\ell 
\frac{\E\left[ \left( \rho\tilde q_{\ell\ell} a_\ell^* Q_\ell \eta_\ell 
- \E[ \rho\tilde q_{\ell\ell} a_\ell^* Q_\ell \eta_\ell ] \right)
\, y_\ell^* Q_\ell u \right]}
{1 - \rho\tilde t_{\ell\ell} a_\ell^* {\mathcal T}_{\ell} a_\ell}
\] 
and proceeding similarly to $\varepsilon_1$, we obtain 
$|\varepsilon_2 | \leq K / \sqrt{n}$, which completes the proof of 
Lemma \ref{lm-E(aQa)}.

\subsection{Proof of Lemma \ref{lm-Deltaj}} 
\label{prf:lm-Deltaj} 
The $F_j$ increase to $F_n = n^{-1} \tr D^{1/2} T A (1 + \delta \tilde D)^{-2}
{\tilde D} A^* T D^{1/2} < 1$ by Lemma \ref{lm:var-bounds}. As $\gamma > 0$ 
and $M_j$ and $G_j$ are increasing, $\bs \Delta_j$ is decreasing. In order 
to show that $\bs\Delta_n = \Delta_n$, we only need to show that 
$M_n + G_n = \rho^2 \tilde\gamma$. We have
\[
G_n = \frac 1n \tr \tilde D ( I + \delta \tilde D)^{-2} A^* T A 
\tilde D ( I + \delta \tilde D)^{-2} A^* T A 
- \frac 1n \sum_{k=1}^n 
\left( \frac{\tilde d_k \, a_k^* T a_k}{(1+\delta\tilde d_k)^2} \right)^2
\]
Recall from \eqref{eq:t-tilde-jj} and \eqref{eq:T-cal(T)_j} that 
$(1 + \delta \tilde d_k)^{-2} a_k^* T a_k = (1 + \delta \tilde d_k)^{-1} 
- \rho \tilde t_{kk}$. Hence
\[
\frac 1n \sum_{k=1}^n 
\left( \frac{\tilde d_k \, a_k^* T a_k}{(1+\delta\tilde d_k)^2} \right)^2
= 
\frac 1n \sum_{k=1}^n \rho^2 \tilde d_k^2 \tilde t_{kk}^2 
- \frac 1n \sum_{k=1}^n 
\frac{\rho \tilde d_k^2 \tilde t_{kk}}{(1+\delta\tilde d_k)} 
+ \frac 1n \sum_{k=1}^n 
\frac{\tilde d_k^2 \, a_k^* T a_k}{(1+\delta\tilde d_k)^3} 
\]
which results in 
\begin{multline*}
M_n + G_n = 
\frac{\rho}{n} \tr \tilde D \tilde T \tilde D ( I + \delta \tilde D)^{-1} 
 - \frac 1n \tr \tilde D^2 ( I + \delta \tilde D)^{-3} A^* T A \\
+ \frac 1n \tr \tilde D ( I + \delta \tilde D)^{-2} A^* T A 
\tilde D ( I + \delta \tilde D)^{-2} A^* T A  .  
\end{multline*}
Now, one can check with the help of \eqref{eq:Ttilde=f(T)} that 
$\rho^2 \tilde\gamma = \rho^2 n^{-1} \tr \tilde D \tilde T\tilde D \tilde T$ 
is equal to the r.h.s.~of this equation. Lemma \ref{lm-Deltaj} is proven.

\end{appendix}

\bibliography{math}

\noindent {\sc Walid Hachem}, {\sc Malika Kharouf} and {\sc Jamal Najim},\\ 
CNRS, T\'el\'ecom Paristech\\ 
46, rue Barrault, 75013 Paris, France.\\
e-mail: \{hachem, kharouf, najim\}@telecom-paristech.fr\\
\\
\noindent {\sc Jack W. Silverstein},\\
Department of Mathematics, Box 8205\\
North Carolina State University\\
Raleigh, NC 27695-8205, USA\\ 
e-mail: jack@math.ncsu.edu\\
\\
\noindent

\end{document}